\documentclass[10pt,reqno]{amsart}

\usepackage{amssymb}
\usepackage{amscd}
\usepackage{amsfonts}
\usepackage{setspace}
\usepackage{version}

\newtheorem{theorem}{Theorem}[section]
\newtheorem{lemma}[theorem]{Lemma}
\newtheorem{proposition}[theorem]{Proposition}
\newtheorem{corollary}[theorem]{Corollary}

\theoremstyle{definition}
\newtheorem{definition}[theorem]{Definition}

\newtheorem{problem}[theorem]{Problem}

\theoremstyle{remark}

\newsavebox{\proofbox}
\savebox{\proofbox}{\begin{picture}(7,7)  \put(0,0){\framebox(7,7){}}\end{picture}}

\newcommand{\md}[1]{\ensuremath{\,(\operatorname{mod}\, #1)}}
\newcommand{\mdsub}[1]{\ensuremath{(\mbox{\scriptsize mod}\, #1)}}
\newcommand{\mdlem}[1]{\ensuremath{(\mbox{\textup{mod}}\, #1)}}

\renewcommand{\leq}{\leqslant}
\renewcommand{\geq}{\geqslant}

\newcommand\im{\operatorname{im}}
\newcommand\Lip{\operatorname{Lip}}
\newcommand\tor{\operatorname{tor}}
\newcommand\sml{\operatorname{sml}}
\newcommand\unf{\operatorname{unf}}

\newcommand\disc{\operatorname{disc}}

\def\R{\mathbb{R}}
\def\C{\mathbb{C}}
\def\Z{\mathbb{Z}}
\def\E{\mathbb{E}}

\def\Q{\mathbb{Q}}
\def\N{\mathbb{N}}
\def\T{\mathbb{T}}
\def\eps{\varepsilon}

\numberwithin{equation}{section}

\begin{document}

\title[Sets of integers with no large sum-free subset]{Sets of integers with no large sum-free subset}


%    author one information

\author{Sean Eberhard}
\address{Centre for Mathematical Sciences, Wilberforce Road, Cambridge CB3 0WA}
\email{s.eberhard@dpmms.cam.ac.uk}

\author{Ben Green}
\address{Centre for Mathematical Sciences, Wilberforce Road, Cambridge CB3 0WA}
\email{b.j.green@dpmms.cam.ac.uk}

\author{Freddie Manners}
\address{Centre for Mathematical Sciences, Wilberforce Road, Cambridge CB3 0WA}
\email{frwm2@cam.ac.uk}

\onehalfspace
%    \subjclass is required.
%\subjclass[2000]{Primary }
%    The 2010 edition of the Mathematics Subject Classification is
%    now available.  If you are citing a classification from the
%    new scheme, use the following input coding instead.
%\subjclass[2010]{Primary }

%    Abstract is required.
\begin{abstract}
Answering a question of P.~Erd\H{o}s from 1965, we show that for every $\eps > 0$ there is a set $A$ of $n$ integers with the following property: every set $A' \subset A$ with at least $\left(\frac{1}{3} + \eps\right) n$ elements contains three distinct elements $x,y,z$ with $x + y = z$.
\end{abstract}

\maketitle

\tableofcontents

\section{Introduction}

\label{sec1}

An old argument of Erd\H{o}s~\cite{erdos1} shows that every set $A$ of $n$ nonzero integers contains a subset $A^\prime\subset A$ of size $|A^\prime|\geq \frac{1}{3}n$ which is \emph{sum-free}, meaning $x+y=z$ has no solutions with $x,y,z\in A^\prime$. The argument is simple: for $\theta\in\R/\Z$ the set $A_\theta$ of $x\in A$ such that $\frac{1}{3}<\{\theta x\}<\frac{2}{3}$ is clearly sum-free, and if $\theta$ is chosen uniformly at random then the expected size of $A_\theta$ is $\frac{1}{3}n$, so $|A_\theta|\geq\frac{1}{3}n$ for some $\theta$.

Let $f(n)$ be the largest $k$ such that every set of $n$ nonzero integers contains a sum-free subset of size $k$. Erd\H{o}s's lower bound $f(n)\geq \frac{1}{3}n$ has not been much improved. As pointed out by Alon and Kleitman~\cite{alonkleitman}, Erd\H{o}s's argument can be modified to show $f(n)\geq \frac{1}{3}(n+1)$: if $\theta\approx 0$ then $A_\theta$ is empty, so for some $\theta$ we must actually have $|A_\theta|>\frac{1}{3}n$, so $|A_\theta|\geq\frac{1}{3}(n+1)$. The best known lower bound is due to Bourgain~\cite{bourgain}, who showed $f(n)\geq\frac{1}{3}(n+2)$ for $n\geq 3$ using an elaborate Fourier-analytic technique. In particular it is unknown whether $f(n)\geq \frac{1}{3}n + \omega(n)$ for some $\omega(n)\to\infty$, though this seems likely.

In the opposite direction, considering the largest element of a subset $A \subset \{1,\ldots,n\}$ gives an obvious upper bound of $f(n)\leq \frac{1}{2}(n+1)$. Improvements to this upper bound have all implicitly used the following device. Suppose that $A$ is a set of size $m$ with no sum-free subset of size larger than $f(m)$ and that $B$ is a set of size $n$ with no sum-free subset of size larger than $f(n)$. Then if $M\in\N$ is sufficiently large $A\cup M B$ is a set of size $m+n$ with no sum-free subset of size larger than $f(m) + f(n)$, so $f(m+n)\leq f(m)+f(n)$. This condition is well known to imply that $f(n)/n$ converges to $\inf f(n)/n$, so to show $f(n)\leq cn + o(n)$ it suffices to find a single set $A$ with no sum-free subset of size larger than $c|A|$. Let $\sigma = \lim f(n)/n = \inf f(n)/n$.

In~\cite{erdos1} Erd\H{o}s mentioned that Hinton proved $\sigma\leq\tfrac{7}{15}\approx0.467$. He also pointed out, attributing the construction to Klarner, that the set $A=\{2,3,4,5,6,8,10\}$ shows $\sigma\leq\tfrac{3}{7}\approx 0.429$. Using a set of size $29$ Alon and Kleitman~\cite{alonkleitman} showed $\sigma\leq\frac{12}{29} \approx 0.414$. Malouf \cite{malouf} in her thesis (as well as Furedi, according to Guy~\cite{guy}) used $A = \{1,2,3,4,5,6,8,9,10,18\}$ to show $\sigma\leq\tfrac{2}{5} = 0.4$. Lewko~\cite{lewko} used a set of size $28$ to show $\frac{11}{28} \approx 0.393$. Incidentally, in a 1992 letter~\cite{erdosklarner} to Klarner, Erd\H{o}s claims this same bound of $\tfrac{11}{28}$, but he includes no proof.

Recently, Alon~\cite{alon} showed that for each $n$ there exists $m$ such that $f(m)/m < f(n)/n$. Thus there is no $n$ such that $f(n)/n = \sigma$. Applying this to Lewko's bound, Alon showed for instance that $\sigma \leq \frac{11}{28} - \eps$ for some $\eps \approx 10^{-50000}$.

The question of whether $\sigma=\tfrac{1}{3}$ has been mentioned several times~\cite{alonspencer,crootlev,erdos1,erdos2,guy,kolountzakis}. Our purpose in this paper is to answer this question affirmatively.

\begin{theorem}\label{mainthm}
There is a set of $n$ positive integers with no sum-free subset of size greater than $\frac{1}{3}n + o(n)$.
\end{theorem}

%Erd\H{o}s~\cite{erdos1} also asked the same question with ``sum-free'' replaced with ``restricted-sum-free'', where a set $A^\prime$ is \emph{restricted-sum-free} if $x+y=z$ has no solutions with $x,y,z\in A^\prime, x\neq y$. In this case Erd\H{o}s offers no upper bound other than $\frac{1}{2}n$, though in~\cite{erdos2} he states that Hilton and Klarner improved the constant $\frac{1}{2}$ here, but gives no further details and we are unable to find a reference for this result. Improving this upper bound seems harder, as all constructions so far have exploited the fact that if $A$ is sum-free and $x\in A$ then $2x\notin A$. As it happens, however, the set $A$ we construct is such that for every subset $A^\prime\subset A$ of size at least $(\frac{1}{3}+\eps)n$ there are $\gg_\eps n^2$ pairs $x,y\in A^\prime$ such that $x+y\in A^\prime$, so such subsets $A^\prime$ are not even restricted-sum-free.

By the above argument it suffices to find, for each $\eps>0$, a single set $A$ with no sum-free subset of size larger than $\left(\tfrac{1}{3}+\eps\right)|A|$. In fact we find a set $A$ such that every subset $A'$ of size larger than $\left(\tfrac{1}{3}+\eps\right)|A|$ contains a solution to $x+y=z$ with $x\neq y$. This answers a further question asked in~\cite{erdos1}.

One of the ingredients of our argument is a rough structure theorem for sets $A$ satisfying conditions of the form $|A-A|<4|A|$, which may be of independent interest. Specifically, if $A$ is a set of integers with $|A-A|\leq(4-\eps)|A|$ then $A$ has density at least $\frac{1}{2} + c\eps$ on some arithmetic progression of length $\gg_{\eps} |A|$. 

\section{Overview of the proof}

\label{sec2}

The proof of Theorem \ref{mainthm} breaks down naturally into several parts.  We outline these informally here, and give an indication of how they combine.

Note that there are certain local obstructions to a set $A$ having the desired property, that is to say having no sum-free subset $A'$ with $|A'| \geq (\frac{1}{3} + \eps)|A|$. For instance, not more than $\left(\tfrac{1}{3} + \eps\right) |A|$ of the elements of $A$ can be odd, as these form a sum-free set in $A$. We think of this as an obstruction coming from $\Z/2\Z$. Not more than $\left(\tfrac{1}{3} + \eps\right)|A|$ of the elements of $A$ can be congruent to $2$ or $3 \md{5}$, an obstruction coming from $\Z/5\Z$. Similarly, not more than $\left(\tfrac{1}{3} + \eps\right) |A|$ elements can be contained in an interval $[x, 2x)$, an obstruction coming from $\R$.

In fact we shall see in Section \ref{sec3} that, in some sense, these restrictions coming from $\Z/Q\Z$ for various $Q$ and from $\R$ are the \emph{only} obstructions to $A$ having the desired property.  

To deal with these restrictions modulo $Q$ and in $\R$ we consider a \emph{weight function} $w : \Z / Q \Z \times [0, 1] \rightarrow (0,\infty)$. Roughly speaking, we will define a set $A \subset \{1,\dots,N\}$ in such a way that a proportion $w(x,y)$ of the elements of $A$ lie near the value $yN$ and are congruent to $x \md{Q}$.  The ``local'' version of our problem is then roughly the following.

\begin{problem}[Local problem]\label{local_problem}
Find $w$ such that if $A \subset \Z / Q \Z \times [0, 1]$ is open and if $\int_A w(x,y) d\mu \geq \frac{1}{3} + \eps$ then $A$ contains a summing triple $x, y, x + y$. Here, $\mu$ denotes the uniform probability measure on $\Z/Q\Z \times [0,1]$.
\end{problem}

In Section \ref{sec3}, we show that if $w$ satisfies a slightly stronger version of Problem \ref{local_problem} (specifically Proposition \ref{nu-exist-3}), then a set $A$ may be constructed from $w$ as suggested above and Theorem \ref{mainthm} holds for this $A$. The actual construction of $A$, which involves a random selection argument, occurs at \eqref{a-prop-1}. A crucial tool in showing that $A$ has the required property (and elsewhere in the paper) is the \emph{arithmetic regularity lemma} due to the second-named author and Tao \cite{green-tao-arithregularity}. The statement of this is recalled in Lemma \ref{ar}.

The remainder of the paper is concerned with constructing the weight function $w$, that is to say with solving the local problem.  This construction is rather involved. At its heart is an iterative argument (Lemma \ref{weight-lem}) allowing us to take a near-solution $w$, such that if $\int_A w(x,y) d\mu \geq \alpha$ then $A$ contains contains many summing triples, and improve it to a nearer-solution $w'$, with corresponding parameter $\alpha'<\alpha$. The sequence $\alpha, \alpha', \alpha'',\dots$ obtained in this way converges rapidly to $\frac{1}{3}+\eps$. 

% The remainder of the paper is concerned with constructing the weight function $w$, that is to say with solving the local problem.  This construction is rather involved. At its heart is an iterative argument (Lemma \ref{weight-lem}) allowing us to take a solution $w$ to the local problem that works for $\eps$, and produce an improved solution $w'$ that works for some $\eps' < \eps$. The derived sequence $\eps,\, \eps',\, \eps'', \dots$ obtained in this way converges rapidly to zero.

The main driver for this iterative argument is a structural result concerning sum-free (or almost sum-free) subsets of $\Z/Q\Z \times [0,1]$ with (uniform) measure just a little more than $\frac{1}{3}$.  The crucial result here is Corollary \ref{cor4.2}, which states that such sets ``avoid zero'', i.e.,~have very little mass on $H \times I$, where $H \leq \Z/Q\Z$ is a subgroup of small index and $I \subset [0,1]$ is a (not too small) open interval containing $0$. Thus if $w$ is chosen to have a lot of its mass concentrated on $H \times I$, then $\int_A w(x, y) d\mu$ is small whenever $\mu(A)$ is a bit more than $\frac{1}{3}$.  The iteration then works in some sense by applying the same arguments to $A \cap (H \times I)$.  In particular, the weight $w$ we construct blows up near zero.

The proof of Corollary \ref{cor4.2} rests on the rather lengthy arguments of Section \ref{sec4}, which concern the structure of open sets $A \subset \Z/Q\Z \times [0,1]$ with $\mu(A - A) \leq 4 \mu(A) - \eps$, where $\mu$ denotes the uniform measure (and generalisations of this statement). The key result here is Corollary \ref{thm:stronger-thmconj}. This in turn is deduced from Theorem \ref{discthm}, which concerns sets of \emph{integers} $A \subset \{1, \dots, N \}$ satisfying the same condition, that $|A - A| \leq 4 |A| - \eps N$. The conclusion is that they have density at least $\frac{1}{2} + c\eps$ on a progression of length $\gg_{\eps} N$, a result which may be of independent interest. (This theme is elaborated upon briefly in Section \ref{sec6}, which is independent of the rest of the paper.) The proof of Theorem \ref{discthm} uses the arithmetic regularity lemma again, as well as an application of the Brunn-Minkowski inequality for open subsets of $\R^2$.

On account of our double application of the arithmetic regularity lemma, the $o(n)$ term in Theorem \ref{mainthm} is more or less ineffective. The authors believe that main obstacle to a more effective $o(n)$ here is the use of the arithmetic regularity lemma in Section \ref{sec4}, which for all we know could be replaced by more elementary arguments.

\emph{Notation.} We will introduce various pieces of notation as we go along. Throughout the paper we will also use the following at least somewhat standard notations.

The expression $O_{A_1,\dots, A_k}(1)$ denotes a constant which may depend on $A_1,\dots, A_k$, and $O_{A_1,\dots, A_k}(Y) = O_{A_1,\dots, A_k}(1)Y$. If we write $X \ll_{A_1,\dots A_k} Y$ then we mean that $X \leq O_{A_1,\dots, A_k}(Y)$. The expression $o_{A_1,A_2,\dots, A_k; N \rightarrow \infty}(1)$ denotes an expression which tends to zero as $N \rightarrow \infty$, the rate at which this happens being possibly dependent on the parameters $A_1,\dots, A_k$. On account of its relative ugliness we will use this notation sparingly.

If $f : \{1,\dots,N\} \to \C$ is a function then we write \[\|f\|_{\ell^p(N)} = \left( \frac{1}{N} \sum_{n \leq N} |f(n)|^p \right)^{1/p}.\] We will use this only when $p = 1$ or $2$. The normalisation, which is perhaps nonstandard, ensures that $\|f\|_{\ell^{1}(N)} \leq \|f\|_{\ell^2(N)} \leq \|f\|_{\infty}$.

\section{The main argument}
\label{sec3}

In this section we prove Theorem \ref{mainthm} assuming the existence of a weight function $w : \Z/Q\Z \times [0,1] \rightarrow (0,\infty)$, the role of which was briefly outlined in the preceding section. Proposition \ref{nu-exist-3} below, whose proof will occupy Sections \ref{sec4} and \ref{sec5}, specifies the properties we shall require of $w$. Before stating this proposition we introduce some pieces of nomenclature. 

We will view both $\Z/q\Z \times [0,1]$ and $\Z/q\Z \times [0,1] \times (\R/\Z)^d$, for various integers $q$ and $d$, as metric spaces. On each of these spaces $X$ we place an ``obvious'' metric, namely a suitable product metric of the discrete metric on $\Z/q\Z$ and the Euclidean metrics on $[0,1]$ and on $(\R/\Z)^d$. This allows us to talk about Lipschitz functions on $X$: if $F : X \rightarrow \C$ then we define $\|F\|_{\Lip}$ to be the infimum of all constants $K$ such that $|F(x) - F(x')| \leq K\, d(x,x')$ for all $x, x' \in X$.

We will also put natural measures on these spaces $X$, which we will always denote by $\mu$ (more precise notation such as $\mu_{\Z/q\Z \times [0,1]}$ would be rather ugly and unnecessary). The measure $\mu$ will always be the product of the uniform probability measures on each factor, namely the uniform measure on $\Z/q\Z$ (which assigns mass $1/q$ to each point), and normalised Lebesgue measure on $[0,1]$ and the torus $(\R/\Z)^d$. 

Finally, if $X$ is one of the sets above and if $\Psi : X \to \C$ is a function then we define 
\begin{equation}\label{t-def}T(\Psi) = \int \Psi(x) \Psi(x') \Psi(x + x') d\mu(x) d\mu(x').\end{equation}
We also write $T(A) = T(1_A)$ if $A \subset X$. We use the same notation when $X = \{1,\dots, N\}$ with the uniform probability measure, so if $f : \{1,\dots, N\} \rightarrow \C$ is a function we write
\begin{equation}\label{t-N-def} T(f) = \frac{1}{N^2} \sum_{n , n' \leq N} f(n) f(n') f(n + n').\end{equation}

By a \emph{weight function} we simply mean a function $w:\Z/Q\Z\times[0,1]\to(0,\infty)$ such that $\int w\,d\mu = 1$. If $\T=(\R/\Z)^d$ and $Q\mid q$ then by $w\times 1_{\T} : \Z/q\Z\times[0,1]\times\T\to(0,\infty)$ we mean the function given by $(w\times1_{\T})(x,y,z) = w(x\md{Q},y)$.

\begin{proposition}\label{nu-exist-3} Let $\eps > 0$. Then there is an integer $Q$ and a Lipschitz weight function $w : \Z/Q\Z \times [0,1] \to (0,\infty)$ with the following property. If $\T = (\R/\Z)^d$ and $Q\mid q$, then for any continuous function $\Psi : \Z/q\Z \times [0,1] \times \T \to [0,1]$ such that $\int \Psi \cdot (w \times 1_{\T}) d\mu \geq \frac{1}{3} + \eps$ we have $T(\Psi) \gg_\eps 1$.
\end{proposition}

Note that because $Q$ and $w$ depend only on $\eps$, there are constants $c_1(\eps)$ and $c_1'(\eps)$ such that $0<c_1(\eps)\leq w(x) \leq c'_1(\eps)$ for every $x\in\Z/Q\Z\times[0,1]$, and there is a constant $L(\eps)$ such that $\|w\|_{\Lip}\leq L(\eps)$. Choose $c_2(\eps)$ to be the implied constant in Proposition \ref{nu-exist-3}, so that the conclusion of the proposition is that $T(\Psi)\geq c_2(\eps)$.

The next lemma is quite standard. In it we encounter the notion of the Gowers $U^2$-norm $\|f\|_{U^2(N)}$, whose definition and relevant basic properties are recalled in Appendix \ref{appA}.

\begin{lemma}\label{random-sample}
Suppose that $p : \{1,\dots,N\} \rightarrow [0,1]$ is a function. Then there is a set $A \subset \{1,\ldots,N\}$ such that $\| 1_A - p \|_{U^2(N)} \ll N^{-1/4}$.
\end{lemma}
\begin{proof}
Choose $A$ at random by including $n$ in $A$ with probability $p(n)$, these choices being independent for different $n$. We claim that $A$ has the required property on average. Write $X_n = 1_A(n) - p(n)$. Then the random variables $X_n$ are independent, bounded by $1$, and of mean zero. We have
\[ \| 1_A - p \|_{U^2(N)}^4 \ll \frac{1}{N^3} \sum_{n_1 + n_2 = n_3 + n_4} X_{n_1} X_{n_2} X_{n_3} X_{n_4}.\]
The expected value of any term on the right hand side with $n_1, n_2, n_3, n_4$ distinct is $0$. This accounts for all except $O(N^2)$ terms, and so $\E \| 1_A - p \|_{U^2(N)}^4 \ll 1/N$. The result follows immediately. 
\end{proof}

Fix $\eps>0$, and let $Q$ and $w$ be as in Proposition \ref{nu-exist-3}. By Lemma \ref{random-sample} there is a set $A$ such that 
\begin{equation}\label{a-prop-1} 1_A(n) = \frac{1}{\|w\|_\infty} w(n\md Q,n/N) + g_{\unf}(n),\end{equation} where $\| g_{\unf} \|_{U^2(N)} = o(1)$. Since $\int w = 1$ and $w$ has Lipschitz constant $O_{\eps}(1)$, it follows from Lemmas \ref{lem-a7} and \ref{gowers-progressions} that
\begin{equation}\label{a-prop-2} \frac{|A|}{N} = \frac{1}{\| w \|_{\infty}} + o_{\eps; N \rightarrow \infty}(1).\end{equation}
We shall show that this set $A$ satisfies Theorem \ref{mainthm}.

\begin{theorem}\label{mainthm-sec3}
Let $N>N_0(\eps)$ be sufficiently large, let $A$ be the set just constructed and suppose $A' \subset A$ has no solutions to $x + y = z$. Then $|A'| \leq (\frac{1}{3} + 2\eps)|A|$.
\end{theorem}
\begin{proof}
Let $A'\subset A$ be a subset of $A$ such that $|A'| \geq (\frac{1}{3} + 2\eps)|A|$. We apply the \emph{arithmetic regularity lemma} \cite{green-tao-arithregularity} to $1_{A'}$. The statement of this lemma is recalled in Lemma \ref{ar}. Let \begin{equation}\label{delta-choice} \delta = \min\left(\frac{c_2(\eps)c_1(\eps)^3}{20c'_1(\eps)^3},\frac{\eps}{4 c_1'(\eps)}, \frac{1}{100}\right)\end{equation} and let $\mathcal{F} : \N \rightarrow \R_+$ be a growth function, depending on $\eps$, to be specified later. Applying the regularity lemma with parameter $\tfrac{1}{8}\delta^4$ and growth function $\mathcal{F}$ we obtain an integer $M\ll_{\eps,\mathcal{F}} 1$ and a decomposition
\begin{equation}\label{reg-decomp} 1_{A'} = f_{\tor} + f_{\sml} + f_{\unf},\end{equation}
where $\|f_{\sml} \|_{\ell^2(N)} \leq \delta^4 / 16$, $\| f_{\unf} \|_{U^2(N)} \leq 1/\mathcal{F}(M)$ and
\begin{equation}\label{f-tor} f_{\tor}(n) = F(n \md{q}, n/N, \theta n)\end{equation}
for some $M$-Lipschitz $F:\Z/q\Z\times[0,1]\times(\R/\Z)^d\to[0,1]$ and some $(\mathcal{F}(M),N)$-irrational $\theta \in (\R/\Z)^d$, where $q,d \leq M$. For the rest of this section, write $\T = (\R/\Z)^d$.

We may regard the $\bmod\,q$ dependence of $f_{\tor}$ as $\bmod\,qQ$ dependence instead without affecting the Lipschitz constant of $F$. Relabelling, we may assume that $Q\mid q$ and $q\ll_\eps M$.

The property that $A'$ is a subset of $A$ manifests as an approximate upper bound for $F$ in terms of the weight $w$. As the next lemma shows, by absorbing the error into $f_{\sml}$ we can assume that $F \leq \|w\|_\infty^{-1} (w \times 1_{\T})$ pointwise.

\begin{lemma}\label{dom}
Suppose $\mathcal{F}$ grows sufficiently rapidly depending on $\eps$ and $N\geq N_0(\eps,\mathcal{F})$ is sufficiently large. Then we can modify the decomposition \eqref{reg-decomp} to $1_{A'} = f'_{\tor} + f'_{\sml} + f_{\unf}$ where $\| f'_{\sml} \|_{\ell^2(N)} \leq \delta^2$ and $f'_{\tor} = F'(n \md{q}, n/N, \theta n)$ for some $O_\eps(M)$-Lipschitz function $F' : \Z/q\Z \times [0,1] \times (\R/\Z)^d \to [0,1]$ such that $F' \leq \|w\|_\infty^{-1} (w \times 1_\T)$ pointwise.
\end{lemma}
\begin{proof}
Let $F' = \min\left(F,\|w\|_\infty^{-1} (w\times 1_\T)\right)$, let $f'_{\tor}(n) = F'(n\md q, n/N,\theta n)$, and let $h = f_{\tor} - f'_{\tor}$. It suffices to prove $\|h\|_{\ell^2(N)}\leq\tfrac{1}{2}\delta^2$. 

Now substituting in the definition of $h$ we have
\begin{align*} \|h\|_{\ell^2(N)}^2
 &  = \E_{n \leq N} h(n)^2 \\ & = \E_{n\leq N} h(n)\left(f_{\tor}(n) - \|w\|_\infty^{-1}(w\times 1_\T)(n\md Q,n/N,\theta n)\right) .\end{align*}
 Recalling that $1_{A'} = f_{\tor} + f_{\sml} + f_{\unf}$ and that
 \[ 1_A(n) = \Vert w \Vert_{\infty}^{-1}(w \times 1_{\T})(n \md{Q}, n/N, \theta n) + g_{\unf}(n),\]we may rewrite this as
 \[ \Vert h \Vert_{\ell^2(N)}^2 = \E_{n \leq N} h(n) \left(-(1_{A}(n) - 1_{A'}(n)) - f_{\sml}(n) + (-f_{\unf}(n) + g_{\unf}(n))\right).\]
 To estimate this, we split into three terms as suggested by the bracketing. The first term is $\leq 0$ since $h \geq 0$ and $1_{A'} \leq 1_A$ pointwise. The second term may be estimated by the Cauchy-Schwarz inequality, remembering that $h \leq 1$ pointwise:
 \[ \E_{n \leq N} h(n) f_{\sml}(n) \leq \Vert f_{\sml} \Vert_{\ell^2(N)} \Vert h \Vert_{\ell^2(N)} \leq \Vert f_{\sml} \Vert_{\ell^2(N)} \leq \tfrac{1}{8}\delta^4.\]
 Finally, the third term is extremely tiny if $\mathcal{F}$ grows quickly enough, by Lemma \ref{gowers-orthog-struct}. Putting all this together gives $\Vert h \Vert_{\ell^2(N)}^2 \leq \frac{1}{4}\delta^4$, and the lemma follows. \end{proof}

Relabelling $f'_{\tor}$ as $f_{\tor}$, $f'_{\sml}$ as $f_{\sml}$ and $F'$ as $F$, we may thus assume that \eqref{reg-decomp} and \eqref{f-tor} hold with $\|f_{\sml}\|_{\ell^2(N)}\leq\delta^2$, $\|f_{\unf}\|_{U^2(N)}\leq1/\mathcal{F}(M)$ and some $O_\eps(M)$-Lipschitz function $F:\Z/q\Z\times[0,1]\times(\R/\Z)^d\to[0,1]$ such that $F\leq\|w\|_\infty^{-1} (w\times1_{\T})$ pointwise, so in other words
\[ F = \|w\|_\infty^{-1} \Psi\cdot(w\times 1_{\T}) \]
for some continuous $\Psi:\Z/q\Z\times[0,1]\times\T\to[0,1]$. By combining this decomposition with Proposition \ref{nu-exist-3} and the counting lemmata in the appendix we can finish the proof of Theorem \ref{mainthm-sec3}.

Using \eqref{reg-decomp}, Cauchy-Schwarz and Lemma \ref{gowers-progressions} we have
\begin{align*}
	\frac{|A'|}{N}
		&= \E_{n \leq N} F(n \md{q}, n/N, \theta n) + \E_{n \leq N} f_{\sml}(n) + \E_{n \leq N} f_{\unf}(n) \\
		&\leq \E_{n \leq N} F(n \md{q}, n/N, \theta n)  + 2\delta.
\end{align*}
Thus by the Lipschitz property of $F$, the irrationality of $\theta$ and Lemma \ref{distribution-integral} we have
\[
	\frac{|A'|}{N} \leq \int F\, d\mu + 3\delta = \frac{1}{\|w\|_\infty} \int\Psi\cdot(w\times 1_{\T})\,d\mu + 3\delta.
\]
If $N>N_0(\eps)$ is large enough then \eqref{a-prop-2} implies that
\[ \frac{|A|}{N} \geq \frac{1}{\|w\|_{\infty}} - \delta.\] Assuming that $|A'| \geq (\frac{1}{3} + 2\eps) |A|$, and recalling that $\delta$ was chosen so that $\delta \leq \eps/4 c_1'(\eps)$, it follows from these two observations that
\[\int \Psi \cdot (w \times 1_{\T}) \, d\mu\geq \tfrac{1}{3} + \eps.\]
Proposition \ref{nu-exist-3} now implies $T(\Psi) \geq c_2(\eps)$, so by the pointwise bounds $c_1(\eps)\leq w \leq c_1'(\eps)$ we have
\[ T(F) = T\left(\|w\|_\infty^{-1}\Psi\cdot(w \times 1_{\T})\right)\geq \frac{c_2(\eps) c_1(\eps)^3} {c'_1(\eps)^3}.\] Write $c_3(\eps)$ for this latter quantity. By Lemma \ref{counting-lemma} it follows that (if $\mathcal{F}$ grows sufficiently rapidly and $N$ is big enough),
\[T(f_{\tor}) \geq T(F) - \tfrac{1}{2}c_3(\eps) \geq \tfrac{1}{2}c_3(\eps).\]
Finally, from Lemma \ref{lem-a9} together with the bounds $\| f_{\sml} \|_{\ell^2(N)} \leq \delta^2$, $\| f_{\unf} \|_{U^2(N)} \leq 1/\mathcal{F}(M)$ and the choice of $\delta$, we conclude that
\[ T(A') = T(f_{\tor} + f_{\sml} + f_{\unf}) \geq \tfrac{1}{4} c_3(\eps).\]
In particular $A'$ has $\gg_{\eps} N^2$ solutions to $x + y = z$ with $x\neq y$. This concludes the proof of Theorem \ref{mainthm-sec3} and hence of Theorem \ref{mainthm} (modulo the results of the next two sections and the appendix).
\end{proof}

\section{Sets of doubling less than 4}\label{sec4}

\def\Sym{\operatorname{D}}
\def\struct{\operatorname{struct}}

In this section we study sets $A$ satisfying $|A - A| \leq (4-\eps)|A|$ or various related but slightly weaker conditions. Our particular aim is to prove Corollary \ref{thm:stronger-thmconj} below, which will be crucial in the construction of the weight function $w$ in Proposition \ref{nu-exist-3}. However, some special cases and corollaries of our main result may be of independent interest and we highlight these in Section \ref{sec6}.

We remind the reader of our convention, mentioned at the beginning of Section \ref{sec3}, that the ``natural'' uniform measure on a space $X$, whether it be $\Z/q\Z\times[0,1]$, $\T$, $\{1,\dots,N\}$, etc., is denoted $\mu$. When we want to refer to the size (i.e., counting measure) of a set $A$, particularly a subset $A \subset \{1,\ldots,N\}$, we will use the notation $|A|$.

If $X$ is a space endowed with a measure $\mu$ (one of the above) then, as usual, we define the \emph{convolution} of two sufficiently nice functions $f_1, f_2 : X \to \C$ by $f_1 \ast f_2(x) = \int f_1(y) f_2(x - y) d\mu(y)$. In the case $X=\{1,\dots,N\}$ we allow $f_1$, $f_2$ and $f_1\ast f_2$ to be defined on $\Z\backslash\{1,\dots,N\}$ as well, but we continue to use the measure $\mu$ which gives each point a mass $1/N$. If $A$ is a set and $t$ is a real number then we define $\Sym_t(A) = \{ x : 1_A \ast 1_{-A}(x) \geq t\}$, the set of ``$t$-popular differences'' of $A$. Note that $\Sym_t(A) \subset A - A$ if $t > 0$.

The main result of this section is the following.

\begin{theorem}\label{discthm} For every $\eps > 0$ there is some $\delta\gg_\eps 1$ such that the following holds. If $A \subset \{1,\dots, N\}$ is a set with $|\Sym_{\delta}(A)| \leq 4|A| - \eps N$ then there is an arithmetic progression $P \subset \{1,\dots, N\}$ of length $|P|\gg_\eps N$ such that $|A \cap P| \geq (\frac{1}{2} + \tfrac{1}{5} \eps)|P|$.
\end{theorem}

The reader may find it helpful to think of the hypothesis $|\Sym_{\delta}(A)|\leq 4|A| - \eps N$ as a slight weakening of $|A-A|\leq 4|A|-\eps N$. To motivate this theorem, we first derive the corollary which will enable us in Section \ref{sec5} to construct a weight function satisfying Proposition \ref{nu-exist-3}.

\begin{corollary}\label{thm:stronger-thmconj} Let $\eps > 0$ and $q  \in \N$. Then there is $\delta\gg_\eps 1$ such that if $A\subset \Z/q\Z\times[0,1]$ is an open set with $\mu(\Sym_{\delta}(A)) \leq 4 \mu(A) - \eps$ then there is a subgroup $H\leq\Z/q\Z$ of index $[\Z/q\Z:H] \ll_{\eps} 1$, an element $x \in \Z/q\Z$, and a subinterval $I$ of $[0,1]$ of length $\mu(I)\gg_{\eps} 1$ such that $A$ has density at least $\tfrac{1}{2} + \tfrac{1}{7}\eps$ on $(x+H)\times I$.
\end{corollary}

\begin{proof}
Let $A \subset \Z/q\Z \times [0,1]$ be an open set such that $\mu(\Sym_{\delta}(A)) \leq 4 \mu(A) - \eps$. Then for some positive integer $K$ depending on $\eps$ and $A$ there is a subset $A'\subset A$, a union of sets of the form $\{a\} \times \left(\frac{i-1}{K}, \frac{i}{K}\right)$, such that $\mu(A')\geq\mu(A) - \frac{1}{32}\eps$. (Note that none of our final quantities can or will depend on $K$.) Then since $A'\subset A$, \[\mu(\Sym_\delta(A')) \leq \mu(\Sym_\delta(A)) \leq 4\mu(A) - \eps \leq 4\mu(A') - \tfrac{7}{8}\eps.\] With an abuse of notation rename $A'$ simply $A$.

For $N$ a large multiple of $q$, consider the map $\pi : \{1,\dots,N\} \rightarrow \Z/q\Z \times [0,1]$ defined by $\pi(n) = (n \md{q}, n/N)$. It is clear (see Lemma \ref{lem-a7}) that for large $N$ the image of $\{1,\dots,N\}$ under $\pi$ is highly equidistributed in $\Z/q\Z \times [0,1]$. In particular we have
\begin{equation}\label{equi}  \E_{n \leq N} \psi(\pi(n)) = \int_{\Z/q\Z \times [0,1]} \psi(x) d\mu(x) + o_{K; N \rightarrow \infty}(1)\end{equation} whenever $\psi$ is ``nice'', in particular whenever $\psi$ has one of the following three forms:
\begin{enumerate}
\item the characteristic function of a union of sets $\{a\} \times \left(\frac{i-1}{K}, \frac{i}{K}\right)$,
\item the characteristic function of the intersection of a set of type (i) with a translate of another set of type (i),
\item a continuous function with Lipschitz constant $K$.
\end{enumerate}
(Note that, conditional on one of these hypotheses, the quantity $o_{K; N \rightarrow \infty}(1)$ is asserted to be independent of $\psi$.)

In particular, if $B = \pi^{-1}(A)$, by case (i) of \eqref{equi} we have\footnote{Note that this would not be true if $A$ were an arbitrary open set, for example if $A$ were a set of small measure containing $\Z/q\Z \times (\Q \cap [0,1])$.}
\begin{equation}\label{limiting-1} \mu(B) = \E_{n \leq N} 1_A(\pi(n)) =  \mu(A) + o_{K; N \rightarrow \infty}(1).\end{equation} 
Furthermore we claim that
\begin{equation}\label{limiting-2}  \mu(\Sym_{2\delta}(B)) \leq \mu(\Sym_{\delta}(A)) + o_{K,\delta; N \rightarrow \infty}(1).\end{equation}  This is a little trickier to justify. First note that by case (ii) of \eqref{equi} that
\begin{align*}
1_B \ast 1_{-B}(n) & = \E_{m \leq N} 1_A (\pi(m))\, 1_A(\pi(m) - \pi(n)) \\ & = \int_{\Z/q\Z \times [0,1]} 1_A (x)\, 1_A (x - \pi(n))d\mu(x) + o_{K; N \rightarrow \infty}(1) \\ & = 1_A \ast 1_{-A} (\pi(n)) + o_{K; N \rightarrow \infty}(1).
\end{align*}
In particular if $N > N_0(K,\delta)$ is large enough then if $n \in \Sym_{2\delta}(B)$ then $\pi(n) \in \Sym_{3\delta/2}(A)$, that is to say if $1_B \ast 1_{-B}(n) \geq 2\delta$ then $1_A \ast 1_{-A}(\pi(n)) \geq 3\delta/2$. Now let $\chi : [0,1] \rightarrow [0,1]$ be a function such that $\chi(x) = 1$ for $x  \geq 3\delta/2$, $\chi(x) = 0$ for $x \leq \delta$, and $\chi$ has Lipschitz constant $O(1/\delta)$. What we have shown implies that if $N > N_0(K,\delta)$ then
\[ \E_{n \leq N} \chi \circ (1_A \ast 1_{-A})(\pi(n)) \geq \mu(\Sym_{2\delta}(B)) .\]
Now $1_A \ast 1_{-A}$ has Lipschitz constant at most $K$, so $\chi \circ (1_A \ast 1_{-A})$ has Lipschitz constant at most $O(K/\delta)$. Thus by case (iii) of \eqref{equi},
\begin{align*} \E_{n \leq N} \chi \circ (1_A \ast 1_{-A})(\pi(n)) & = \int_{\Z/q\Z \times [0,1]} \chi \circ (1_A \ast 1_A)(x) d\mu(x) + o_{K,\delta; N \rightarrow \infty}(1)\\ & \leq \mu (\Sym_{\delta}(A)) + o_{K,\delta; N \rightarrow \infty}(1).\end{align*} This completes the justification of the claim \eqref{limiting-2}. 

Comparing \eqref{limiting-1} and \eqref{limiting-2} and recalling the hypothesis that $\mu(\Sym_{\delta}(A)) \leq 4 \mu(A) - \frac{7}{8}\eps$, we see that if $N >N_0(K,\eps,\delta)$ is large enough then $|\Sym_{2\delta}(B)| \leq 4|B| - \frac{5}{6}\eps N$.  Choose $\delta \gg_\eps 1$ small enough that Theorem \ref{discthm} holds with $2\delta$ in place of $\delta$ and $\frac{5}{6}\eps$ in place of $\eps$. Then there is a progression $P \subset \{1,\dots,N\}$ of length $L=|P|\gg_{\eps} N$, say $P = \{x_0 + \lambda d : \lambda = 0,1,\dots, L-1\}$, such that $|B \cap P| \geq (\frac{1}{2} + \tfrac{1}{6}\eps)|P|$.

It is readily seen that the image $\pi(P)$ is highly equidistributed (as $N \rightarrow \infty$) on $\pi(x_0) + H \times I$, where $H \leq \Z/q\Z$ is the subgroup of index $\gcd(q,d)\leq d \ll_{\eps} 1$, and $I = \left[0,\, \frac{d L}{N} \right]$ has length $\frac{d L}{N} \gg_{\eps, \alpha} 1$, so by a variant of \eqref{equi}, case (i), we have \[\frac{|B\cap P|}{|P|} = \frac{\mu(A \cap (\pi(x_0) + H \times I))}{\mu(\pi(x_0) + H\times I)} + o_{\eps, K; N \to \infty}(1).\] Therefore, if $N$ is large enough depending on $\eps$ and $K$, \[\mu(A \cap (\pi(x_0) + H \times I)) \geq (\tfrac{1}{2} + \tfrac{1}{7} \eps) \mu(\pi(x_0) + H \times I).\qedhere\]
\end{proof}

We devote the rest of this section to the proof of Theorem \ref{discthm}. The argument uses several nontrivial ingredients: the arithmetic regularity lemma (Lemma \ref{ar}) again, a ``stability'' version of Kemperman's theorem due to Tao~\cite{tao-blog},\cite[Section 3.2]{tao-kemperman} and the Brunn-Minkowski theorem. We begin with the regularity lemma. 
Let the hypotheses be as in Theorem \ref{discthm}, thus $A \subseteq \{1,\dots,N\}$ is a set with $|\Sym_{\delta}(A)| \leq 4|A| - \eps N$. Let $\mathcal{F} : \N \rightarrow \R_+$ be a growth function depending on $\eps$ to be chosen later. Let $\tilde\eps = \min\left(\eps,\tfrac{1}{1000}\right)$. Then there is some $M \ll_{\eps,\mathcal{F}}1$ such that \[ 1_A = f_{\tor} + f_{\sml} + f_{\unf},\]where $\| f_{\sml} \|_{\ell^2(N)} \leq \tilde\eps^{10}$, $\| f_{\unf} \|_{U^2(N)} \leq 1/\mathcal{F}(M)$ and \[f_{\tor} = F(n \md{q}, n/N, \theta n)\] for some $F : \Z/q\Z \times [0,1] \times(\R/\Z)^d \to [0,1]$ such that $q,d,\|F\|_{\Lip} \leq M$ and for some $(\mathcal{F}(M),N)$-irrational $\theta\in(\R/\Z)^d$. As usual we abbreviate $(\R/\Z)^d$ to $\T$.

Let $\tilde M = \lceil{\tilde\eps}^{-10} M\rceil$ and consider, for $a \in \Z/q\Z$ and $i \in \{1,\dots, \tilde M\}$, the progressions \[I_{a,i} = \left\{ n \in  \left( \frac{(i-1)N}{\tilde M}, \frac{iN}{\tilde M} \right]: n \equiv a\ \md{q} \right\}.\] Define $F_{a,i}:\T\to[0,1]$ by $F_{a,i}(x) = F(a,i/{\tilde M},x)$. Then because $F$ is $M$-Lipschitz, $F_{a,i}$ is $M$-Lipschitz and $f_{\tor}$ differs by at most ${\tilde\eps}^{10}$ from a function $f_{\struct}$ which we define by
\[ f_{\struct}(n) = \sum_{a \mdsub{q}} \sum_{i=1}^{\tilde M} 1_{I_{a,i}}(n) F_{a,i}(\theta n). \]
Absorbing the error of $\tilde\eps^{10}$ into $f_{\sml}$, we have a decomposition
\[ 1_A = f_{\struct} + f'_{\sml} + f_{\unf}\]
where $\|f'_{\sml}\|_{\ell^2(N)} \leq 2{\tilde\eps}^{10}$ and $\|f_{\unf}\|_{U^2(N)} \leq 1/\mathcal{F}(M)$. Now given an arbitrary growth function $\tilde{\mathcal{F}}$ depending on ${\eps}$, we may choose $\mathcal{F}$ to grow sufficiently rapidly depending on ${\eps}$ so that $\mathcal{F}(M) \geq\tilde{\mathcal{F}}(\tilde M)$, whence $\|f_{\unf}\|_{U^2(N)} \leq 1/{\tilde{\mathcal{F}}}({\tilde M})$ and $\theta$ is $(\tilde{\mathcal{F}}(\tilde{M}),N)$-irrational. Clearly we may now rename $\tilde M$ as $M$, $f'_{\sml}$ as $f_{\sml}$ and $\tilde{\mathcal{F}}$ as $\mathcal{F}$, so that
\begin{equation}\label{arith-decomp} 1_A = f_{\struct} + f_{\sml} + f_{\unf},\end{equation}
where $\|f_{\sml}\|_{\ell^2(N)}\leq 2{\tilde\eps}^{10}$, $\|f_{\unf}\|_{U^2(N)} \leq1/\mathcal{F}(M)$, 
\[ f_{\struct}(n) = \sum_{a \mdsub{q}} \sum_{i=1}^{M} 1_{I_{a,i}}(n) F_{a,i}(\theta n),\]
and \[ I_{a,i} = \left\{ n \in  \left(\frac{(i-1)N}{M}, \frac{iN}{M}\right]: n \equiv a \md{q} \right\}.\]

Write $\alpha(a,i)$ for the density of $A$ on $I_{a,i}$. We will show that $\alpha(a,i)\geq\tfrac{1}{2} + \tfrac{1}{5}\eps$ for some $(a,i)$. Note that while $|I_{a,i}|$ need not be exactly $N/qM$, at worst it differs from $N/qM$ by $2$. We will deal with this small discrepancy taking $N\geq N_0(\eps)$ sufficiently large depending on $\eps$. This is acceptable: if $N<N_0(\eps)$ then Theorem \ref{discthm} is trivially satisfied by taking $P$ to be a suitable singleton.\footnote{Alternatively, one could arrange that $N$ is always multiple of $qM$, in which case $|I_{a,i}|$ is exactly $N/qM$.}

We proceed by examining how the behaviour of $1_A$ is modelled by the more ``structured'' functions $F_{a,i}(\theta n)$, which in view of the decomposition \ref{arith-decomp} involves estimating the effect of $f_{\sml}$ and $f_{\unf}$. The term $f_{\sml}$ is the more troublesome of the two. The following simple lemma is useful here.

\begin{lemma}\label{setE}
For all $(a,i)\in\Z/q\Z\times\{1,\dots,M\}$ outside an exeptional subset $E$ of size at most ${\tilde\eps}^4 qM$ we have
$\E_{n\in I_{a,i}} |f_{\sml}(n)| \leq {\tilde\eps}^5$.
\end{lemma}
\begin{proof}
If this were not the case we would have \[\E_{n\leq N} |f_{\sml}(n)| > \frac{1}{N}\left(\frac{N}{qM} - 2\right) qM{\tilde\eps}^9 \geq 2{\tilde\eps}^{10},\] whence by Cauchy-Schwarz $\|f_{\sml}\|_{\ell^2(N)} > 2{\tilde\eps}^{10}$, a contradiction.
\end{proof}

\begin{lemma}\label{lem0.5}
Let $E$ be as in the preceding lemma. For all $(a,i)\in\Z/q\Z\times\{1,\dots,M\}$ outside $E$ we have $\int_{\T} F_{a,i} \geq \alpha(a,i) - {\tilde\eps}^4$. %Hence \[ |A| = \frac{N}{qM}\sum_{a \mdsublem{q}}\sum_{i = 1}^M \alpha(a,i) \leq \frac{N}{qM}\sum_{a \mdsublem{q}}\sum_{i = 1}^M \int_{\T} F_{a,i} + 2{\tilde\eps}^4 N.\]
\end{lemma}
\begin{proof}
By Lemma \ref{gowers-progressions} the average of $f_{\unf}$ over any progression $I_{a,i}$ is less than $\tfrac{1}{3}{\tilde\eps}^4$ provided that $\mathcal{F}$ grows sufficiently rapidly, and by Lemma \ref{setE} for all $(a,i)\notin E$ the average of $f_{\sml}$ on $I_{a,i}$ is also at most $\tfrac{1}{3}{\tilde\eps}^4$. Thus if $(a,i)\notin E$ we have
\[ \alpha(a,i) = \E_{n \in I_{a,i}} 1_A(n) \leq \E_{n \in I_{a,i}} F_{a,i}(\theta n) + \tfrac{2}{3}{\tilde\eps}^4  \leq \int_{\T} F_{a,i} + {\tilde\eps}^4.\]
Assuming $\mathcal{F}$ grows sufficiently rapidly, the last step follows from the $(\mathcal{F}(M),N)$-irrationality of $\theta$ and Lemma \ref{distribution-integral-a}. 
 \end{proof}

We need a slightly technical lemma concerning level sets of Lipschitz functions.

\begin{lemma}\label{soft-thresholding}
Let $\eta > 0$. If $\mathcal{F}$ grows sufficiently quickly depending on $\eta$ then the following is true. If $F : \T \rightarrow [0,1]$ is $M$-Lipschitz, $\theta$ is $(\mathcal{F}(M),N)$-irrational and $I \subset \{1,\dots,N\}$ is any progression of length at least $N/M^2$, then the proportion of $n \in I$ such that $F(n\theta) > \eta$ is at least $\mu(\{x \in \T : F(x) > 2\eta\}) - \eta$.
\end{lemma}
\begin{proof}
We want to compute $\E_{n \in I} \chi \circ F(n\theta)$, where $\chi$ is the cutoff $1_{x \geq \eta}$. Replace $\chi$ by a function $\tilde\chi$ with $\| \tilde\chi\|_{\Lip} \ll 1/\eta$ such that $\tilde\chi(x) = 0$ for $x < \eta$ and $\tilde\chi(x) = 1$ for $x \geq 2\eta$. Then $\E_{n \in I} \chi \circ F(n\theta) \geq \E_{n \in I} \tilde\chi \circ F(n\theta)$. However the function $\tilde\chi \circ F$ is Lipschitz with $\| \tilde\chi \circ F\|_{\Lip} \ll M/\eta$ and so, if $\mathcal{F}$ grows sufficiently rapidly, since $\theta$ is so irrational, Lemma \ref{distribution-integral-a} implies that $\E_{n \in I}\tilde\chi \circ F (n\theta) \geq  \int_{\T} \tilde\chi \circ F - \eta$. On the other hand the integral here is at least the measure of $\{x : F(x) \geq 2\eta\}$.
\end{proof}

The following lemma has more meat to it and is a crucial ingredient of our argument. It encodes the fact that if $B_1, B_2$ are open subsets of a torus then the measure $\mu(B_1 + B_2)$ is at least $\min(\mu(B_1) + \mu(B_2),1)$, a 1953 result due to Macbeath \cite{macbeath}. More accurately, we require a ``robust'' version of this result which was obtained in \cite[Proposition 6.1]{green-ruzsa}, and recently given the following elegant formulation by Tao \cite{tao-kemperman}: if $S_1,S_2\subset\T$ are open and $0\leq t\leq\min(\mu(S_1),\mu(S_2))$ then \begin{equation}\label{robustmacbeath}\int_{\T} \min(1_{S_1}\ast 1_{S_2},t)\,d\mu \geq t\min(\mu(S_1) + \mu(S_2) - t, 1).\end{equation}

\begin{lemma}\label{lipschitz-sum}
Let $0<\eta<1$ and suppose that $F_1, F_2 : \T \rightarrow [0,1]$ are $M$-Lipschitz functions such that $\int F_1, \int F_2 \geq 2\eta^{1/6}$. Then the measure of the set of $x$ for which $F_1 \ast F_2(x) \geq \eta$ is at least $\min\left(\int F_1 + \int F_2, 1\right) - 4\eta^{1/6}$.
\end{lemma}
\begin{proof}
Let $S_i = \{ x : F_i(x) > \eta^{1/3}\}$ for $i = 1,2$. Clearly $\mu(S_i) \geq \int F_i - \eta^{1/3}$, so in particular $\mu(S_1), \mu(S_2)\geq \eta^{1/6}$. By \eqref{robustmacbeath} we therefore have
\[ \int_{\T} \frac{\min(1_{S_1} \ast 1_{S_2}(x), \eta^{1/6})}{\eta^{1/6}} dx \geq  \min(\mu(S_1) + \mu(S_2) - \eta^{1/6}, 1).\]
Writing $X$ for the set of $x\in\T$ such that $1_{S_1} \ast 1_{S_2}(x) \geq \eta^{1/3}$, the left-hand side here is bounded by $\mu(X) + \eta^{1/6}$, so $\mu(X) \geq \min(\int F_1 + \int F_2, 1) - 4\eta^{1/6}$. On the other hand, for $x \in X$ we certainly have $F_1 \ast F_2(x) \geq \eta^{2/3} 1_{S_1}\ast1_{S_2}(x)\geq \eta$.
\end{proof}

\begin{lemma}\label{lem0.6}
If $(a,i),(a',i')\notin E$ and $\alpha(a,i)$, $\alpha(a',i') \geq 2{\tilde\eps}^2$ then 
\[ |\Sym_{{\tilde\eps}^{20}/10M^2}(A) \cap I_{a-a',i-i'}| \geq \frac{N}{qM} \min( \alpha(a,i) + \alpha(a',i'), 1)- \frac{10{\tilde\eps}^2 N}{qM},\]
and the same bound holds for $|\Sym_{{\tilde\eps}^{20}/10M^2}(A) \cap I_{a-a',i-i'+1}|$.
\end{lemma}

If $f$ is a function on an abelian group we write $f^\circ$ for the function $f^\circ(x) = f(-x)$.

\begin{proof}
Dealing with $I_{a-a',i-i'}$ and $I_{a-a',i-i'+1}$ are similar, so we focus on the former. By Lemma \ref{lem0.5} then, it suffices to prove 
\[ |\Sym_{{\tilde\eps}^{20}/10M^2}(A) \cap I_{a-a',i-i'}| \geq \frac{N}{qM} \min\left( \int F_{a,i} + \int F_{a',i'}, 1\right)- \frac{8{\tilde\eps}^2 N}{qM}\] for $(a,i)$ and $(a',i')$ outside $E$ and such that $\int F_{a,i}, \int F_{a',i'} \geq {\tilde\eps}^2$.

For all except maybe $2{\tilde\eps}^2 N/qM$ values of $d \in I_{a-a',i-i'}$ (those near the left ends),
\begin{equation}\label{not-endpoint} |I_{a,i} \cap (d + I_{a',i'})| \geq \frac{{\tilde\eps}^2 N}{qM}, \end{equation}
and for any such $d$ we have, if $\mathcal{F}$ is sufficiently rapidly growing,
\begin{align}
\nonumber f_{\struct}|_{I_{a,i}} \ast f^{\circ}_{\struct}|_{I_{a',i'}}(d)
 & = \sum_{n \in I_{a,i} \cap(d + I_{a',i'})} F_{a,i}(\theta n)\, F_{a',i'}(\theta (d+n)) \\
 & \geq |I_{a,i} \cap (d + I_{a',i'})| \left(F_{a,i} \ast F^{\circ}_{a',i'}(\theta d) - \tfrac{1}{4}{\tilde\eps}^{12}\right).\label{eq7}
\end{align}
Here we used the $(\mathcal{F}(M),N)$-irrationality of $\theta$, Lemma \ref{distribution-integral-a} and the fact that the product of two $M$-Lipschitz functions, each of which is bounded pointwise by $1$, is $2M$-Lipschitz. By Lemma \ref{lip-conv}, $F_{a,i}\ast F^\circ_{a',i'}$ is also $M$-Lipschitz, so again by the $(\mathcal{F}(M),N)$-irrationality of $\theta$ and by Lemma \ref{soft-thresholding} the proportion of $d\in I_{a-a',i-i'}$ such that $F_{a,i}\ast F^{\circ}_{a',i'}(\theta d) \geq \tfrac{1}{2}{\tilde\eps}^{12}$ is at least $\mu(Y) - {\tilde\eps}^{12}$, where \[Y=\{y : F_{a,i}\ast F^\circ_{a',i'}(y)\geq{\tilde\eps}^{12}\}.\] But by Lemma \ref{lipschitz-sum} with $\eta={\tilde\eps}^{12}$, $\mu(Y) \geq \min\left(\int F_{a,i} + \int F_{a',i'},1\right) - 4{\tilde\eps}^2$. Putting this all together, \[f_{\struct}|_{I_{a,i}}\ast f^\circ_{\struct}|_{I_{a',i'}}(d) \geq \frac{{\tilde\eps}^{14} N}{4qM}\] for a set of $d\in I_{a-a',i-i'}$ of size at least \[\frac{N}{qM} \min\left(\int F_{a,i} + \int F_{a',i'},\, 1\right) - \frac{7{\tilde\eps}^2 N}{qM}.\]

Now by Lemma \ref{setE} and Young's inequality (Lemma \ref{l1-conv}) we can absorb the contribution of $f_{\sml}$ and conclude that 
\[ (f_{\struct} + f_{\sml})|_{I_{a,i}} \ast (f_{\struct} + f_{\sml})^{\circ}_{I_{a',i'}} (d) \geq \frac{{\tilde\eps}^{14} N}{5qM}\] for these same values of $d$. Finally we add in the contribution of $f_{\unf}$. Recalling from \eqref{arith-decomp} that $1_A = f_{\struct} + f_{\sml} + f_{\unf}$, Lemma \ref{gowers-error} implies that
\[ 1_A|_{I_{a,i}} \ast 1_{-A}|_{I_{a',i'}}(d) \geq \frac{{\tilde\eps}^{14} N}{8qM}\] for all $d$ in a subset $I_{a-a',i-i'}$ of size at least
\[\frac{N}{qM} \min\left(\int F_{a,i} + \int F_{a',i'},\, 1\right) - \frac{8{\tilde\eps}^2 N}{qM}.\]
All these $d$ lie in $\Sym_{{\tilde\eps}^{14}/8qM}(A)$, which is of course contained in $\Sym_{{\tilde\eps}^{20}/10M^2}(A)$.\end{proof}

To use the bound supplied by the preceding lemma we apply the Brunn-Minkowski theorem, which states that if $X, Y \subset \R^d$ are open then $\mu(X + Y)^{1/d} \geq \mu(X)^{1/d} + \mu(Y)^{1/d}$. We require the case $d = 2$. For a wider discussion and proof, see \cite{brunn-minkowski-survey}.

\begin{lemma}\label{doubling-4}
Given a function $\alpha : \Z/q\Z \times \{1,\dots,M\} \to [0,1]$ and $(x,y) \in \Z/q\Z\times \{-M,\dots, M\}$, define $\tilde \alpha(x,y) = \max (\alpha(a,i) + \alpha(a',i'))$, where the maximum is taken over all $(a,i), (a',i')\in\Z/q\Z\times\{1,\dots,M\}$ such that $\alpha(a,i)$, $\alpha(a',i')>0$, $a-a' = x$ and either $i-i' = y$ or $i - i' + 1 = y$. Then
\[ \sum_{x,y} \tilde \alpha(x,y) \geq 4 \sum_{a,i} \alpha(a,i).\]
\end{lemma}
\begin{proof}
Consider the open sets $X, X' \subset \Z/q\Z\times\R^2$ defined by
\[ X = \bigcup_{(a,i) \in \Z/q\Z \times\{1,\dots,M\}} \{a\} \times (i-1, i) \times (0, \alpha(a,i)),\]
\[ X' = \bigcup_{(a',i') \in \Z/q\Z \times\{1,\dots,M\}} \{a'\} \times (i'-1, i') \times (-\alpha(a',i'),0).\]
Note that
\begin{align*} X - X' & = \bigcup_{(a,i), (a',i')} \{a - a'\} \times (i - i' - 1, i - i' + 1) \times (0, \alpha(a,i) + \alpha(a', i')) \\ 
& = \bigcup_{(x,y) \in \Z/q\Z \times \{1,\dots,M\}} \{x\} \times (y-1, y) \times (0, \tilde\alpha(x,y)).\end{align*}
where in the last equality we have ignored a set of measure zero. Thus, if $\nu$ is the product of counting measure on $\Z/q\Z$ and Lebesgue measure $\lambda$ on $\R^2$, we have $\nu(X) = \nu(X') = \sum_{a,i} \alpha(a,i)$ and $\nu(X - X') = \sum_{x,y} \tilde\alpha(x,y)$.
It therefore suffices to show that \[\label{bm}  \nu(X - X') \geq 4\nu(X).\]

The case $q = 1$ of this is immediate from the Brunn-Minkowski inequality. A simple argument allows us to extend this to general $q$. Indeed, let $X_a, X'_a$ be the fibres of $X, X'$ respectively above $a \in\Z/q\Z$. Then $X_a, X'_a$ are open subsets of $\R^2$. Pick $a$ such that $\lambda(X_a) = \lambda(X'_a) = \sum_i \alpha(a,i)$ is largest. If $X_a\neq\emptyset$ then the Brunn-Minkowski inequality implies that
\[   \lambda(X_a - X'_{a_\ast}) \geq \left(\lambda(X_a)^{1/2} + \lambda(X_{a_\ast})^{1/2}\right)^2 \geq 4\lambda(X_a). \]
However the sets $X_a - X'_{a_\ast}$ are disjoint as $a$ ranges over $\Z/q\Z$, since each lies in a different fibre over $\Z/q\Z$. Therefore
\[ \nu(X-X') \geq \sum_{a: X_a\neq\emptyset} \lambda(X_a - X'_{a_\ast}) \geq \sum_{a:X_a\neq\emptyset} 4\lambda(X_a) = 4\lambda(X).\qedhere\]
\end{proof}

In fact we need the following more robust variant of the above, easily deduced from it.

\begin{lemma}\label{doubling-4-witheps}
Let $\eta>0$. Given a function $\alpha : \Z/q\Z \times\{1,\dots,M\} \to [0,1]$ and $(x,y) \in \Z/q\Z\times \{-M,\dots, M\}$, define $\tilde \alpha(x,y) = \max (\alpha(a,i) + \alpha(a',i'))$, where the maximum is taken over all $(a,i), (a',i')\in\Z/q\Z\times\{1,\dots,M\}$ such that $\alpha(a,i)$, $\alpha(a',i')>\eta$, $a-a' = x$ and either $i-i' = y$ or $i - i' + 1 = y$. Then
\[ \sum_{x,y} \tilde \alpha(x,y) \geq 4 \sum_{a,i} \alpha(a,i) - 4\eta qM.\]
\end{lemma}
\begin{proof}
Let \[\alpha^\dag(a,i) = \begin{cases}\alpha(a,i) & \text{if $\alpha(a,i)>\eta$,}\\ 0 &\text{otherwise.}\end{cases}\] Then if we define, as in Lemma \ref{doubling-4}, $\tilde\alpha^\dag(x,y) = \max(\alpha^\dag(a,i) + \alpha^\dag(a',i'))$, where the maximum is taken over all $(a,i), (a',i')$ such that $\alpha^{\dag}(a,i), \alpha^{\dag}(a',i') > 0$, $a - a' = x$ and either $i - i' = y$ or $i - i' + 1 = y$, then $\tilde\alpha^\dag = \tilde\alpha$ as defined above. It follows then from Lemma \ref{doubling-4} that \[\sum_{x,y} \tilde\alpha(x,y) = \sum_{x,y} \tilde\alpha^\dag(x,y) \geq 4\sum_{a,i} \alpha^\dag(a,i) \geq 4\sum_{a,i}\alpha(a,i) - 4\eta qM.\qedhere\]
\end{proof}

Now we are ready to put everything together and complete the proof of Theorem \ref{discthm}. Let $\delta = {\tilde\eps}^{20}/10M^2$. Then certainly $\delta \gg_{\eps} 1$. Recall that $\alpha(a,i)$ is the density of $A$ on $I_{a,i}$. Define \[\alpha'(a,i) = \begin{cases} \alpha(a,i) & \text{if $(a,i)\notin E$,}\\ 0 & \text{if $(a,i)\in E$.}\end{cases}\] Then Lemma \ref{lem0.6} may be rephrased as follows: if $\alpha'(a,i), \alpha'(a',i') \geq 2{\tilde\eps}^2$ then \[ |\Sym_{\delta}(A) \cap I_{a - a', i - i'}| \geq \frac{N}{qM} \min (\alpha'(a,i) + \alpha'(a',i'), 1) - \frac{10{\tilde\eps}^2 N}{qM},\] with the same bound for $|\Sym_{\delta}(A) \cap I_{a - a', i - i' + 1}|$. It follows that  \[ |\Sym_{\delta}(A)| \geq \frac{N}{qM} \sum_{x,y} \min(\tilde{\alpha}'(x,y),1) - 10{\tilde\eps}^2 N,\] where $\tilde{\alpha}'$ is as defined from $\alpha'$ as in Lemma \ref{doubling-4-witheps} with $\eta = 2{\tilde\eps}^2$. Recalling that $\tilde\eps = \min(\eps,\frac{1}{1000})$, this implies \[|\Sym_{\delta}(A)| \geq \frac{N}{qM} \sum_{x,y} \min(\tilde{\alpha}'(x,y),1 + \tfrac{2}{5}{\eps}) - \tfrac{9}{10}{\eps} N.\] Supposing that $\tilde\alpha'(x,y)<1+\tfrac{2}{5}{\eps}$ for all $(x,y)$, Lemma \ref{doubling-4-witheps} implies
\[ |\Sym_\delta(A)| > \frac{4N}{qM} \sum_{a,i} \alpha'(a,i) - \tfrac{99}{100}{\eps} N > \frac{4N}{qM} \sum_{a,i} \alpha(a,i) - \tfrac{999}{1000} \eps N > 4|A| - \eps N.\] Thus if $|\Sym_\delta(A)|\leq 4|A|-{\eps} N$, there must be some $(x,y)$ such that $\tilde\alpha'(x,y)\geq 1 + \tfrac{2}{5}{\eps}$, whence for some $(a,i)$ we must have $\alpha(a,i)\geq\tfrac{1}{2} + \tfrac{1}{5}{\eps}$. This completes the proof of Theorem \ref{discthm}.

\section{Construction of the weight function}

\label{sec5}

In this section we prove Proposition \ref{nu-exist-3} by constructing an appropriate weight function $w$. The reader may wish to take this opportunity to recall the statement of that result. A key ingredient in the proof is the following corollary of the results of the Section \ref{sec4}. It states that an ``almost sum-free'' open subset of $\Z/q\Z \times [0,1]$ with density larger than $\frac{1}{3}$ must ``avoid the origin''. Recall that $\mu$ is the natural probability measure on $\Z/q\Z \times [0,1]$, namely the product of the uniform measure on $\Z/q\Z$ and the Lebesgue measure. 

\begin{corollary}\label{cor4.2}
Let $\eps,\eta > 0$ and $q \in \N$. Suppose that $A\subset\Z/q\Z\times[0,1]$ is an open set with $\mu(A)\geq\frac{1}{3} + \eps$ and $T(A)\leq\eta$. Then $\mu(A\cap (H\times I)) \ll_{\eps} \eta$ for some subgroup $H\leq\Z/q\Z$ of index $\ll_{\eps} 1$ and some interval $I$ around $0$ of length $\gg_{\eps} 1$.
\end{corollary}
\begin{proof} We may assume that $\eta \leq \eta_0(\eps)$, for some $\eta_0(\eps)$ to be specified later. If not, the corollary is trivial by taking $H = \Z/q\Z$ and $I = [0,1]$. Let $\delta\gg_\eps 1$ be as in Corollary \ref{thm:stronger-thmconj}. Recall that $\Sym_{\delta}(A) = \{ x: 1_A\ast1_{-A}(x)\geq\delta\}$ and first suppose that $\mu(\Sym_{\delta}(A))> 4\mu(A)-\eps$. Write $\Sym_{\delta}(A)_+ = \Sym_{\delta}(A) \cap [0,1]$. Since $\Sym_{\delta}(A)$ is symmetric about $0$, we have $\mu(\Sym_{\delta}(A)_+) >  2\mu(A) - \frac{1}{2}\eps$. It follows that 
$\mu(A) + \mu(\Sym_{\delta}(A)_+) > 3 \mu(A) - \frac{1}{2}\eps > 1 +2\eps$, and so by the pigeonhole principle $\mu(A \cap \Sym_{\delta}(A)_+) > 2\eps$. This implies that $T(A) \geq 2\eps \delta \gg_{\eps} 1$. If $\eta_0(\eps)$ is small enough then this is more than $\eta$, and the corollary is established in this case.

The other possibility is that $\mu(\Sym_{\delta}(A)) \leq 4\mu(A) - \eps$. In this case, by Theorem \ref{thm:stronger-thmconj} there is a subgroup $H \leq \Z/q\Z$ of index $m \ll_\eps 1$ and an interval $I \subset [0,1]$ of length $\ell \gg_{\eps} 1$ such that $A$ has density at least $\frac{1}{2} + \tfrac{1}{7}\eps$ on $(x + H) \times I$, for some $x \in \Z/q\Z$. Let $\eps' = \tfrac{1}{7}\eps \ell$, and suppose that $(h, t) \in H \times [0,\eps']$. Then both $A \cap ((x + H) \times I)$ and $\big( A \cap ((x + H) \times I)\big) + (h, t)$ lie in $(x + H) \times I'$, where $I' = I + [0,\eps']$. Noting that $\mu(H \times I) = \ell/m$ and $\mu(H \times I') = (\ell + \eps')/m$, we have \[\mu(A \cap (A + (h, t)))  \geq 2\, \left(\tfrac{1}{2} + \tfrac{1}{7}\eps\right) \mu(H \times I) - \mu(H \times I') \geq \frac{\eps'}{m}.\] It follows that $T(A) \geq \mu(A \cap (H \times [0,\eps'])) \frac{\eps'}{m}$. Since $T(A) \leq \eta$, this implies that $\mu(A \cap (H \times [0,\eps'])) \leq \eta m/\eps' \ll_{\eps} \eta$, and we have proved the corollary in this case too.
\end{proof}

Using the above corollary, we can construct a weight function on $\Z/Q\Z\times[0,1]$ for some $Q\ll_\eps 1$, packing most of its weight near $0$ in a certain sense, and prove that it satisfies Proposition \ref{nu-exist-3}. Our iterative strategy\footnote{This strategy was suggested to us by the proof of the contraction mapping theorem.} is embodied in the following lemma.

\begin{lemma}\label{weight-lem} Let $\eps>0$ and suppose that $\alpha\geq\frac{1}{3} + \frac{1}{8}\eps$. Suppose we are given a Lipschitz weight function $w:\Z/Q\Z\times[0,1]\to(0,\infty)$ and $\eta>0$ such that if $Q\mid q$ and $A\subset \Z/q\Z\times[0,1]$ is an open set such that $T(A)\leq\eta$ then \[\int_A w\,d\mu\leq\alpha.\] Then for some $Q'$ there is a Lipschitz weight function $w':\Z/Q'\Z\times[0,1]\to(0,\infty)$ and $\eta'>0$ such that if $Q'\mid q'$ and $A\subset\Z/q'\Z\times[0,1]$ is an open set such that $T(A)\leq\eta'$ then \[\int_A w'\,d\mu\leq\alpha',\] where $\alpha'=\tfrac{3}{4}\alpha + \tfrac{1}{4}\left(\tfrac{1}{3} + \tfrac{1}{8}\eps\right)$.
\end{lemma}
\begin{proof}
Apply Corollary \ref{cor4.2} with $\frac{1}{8}\eps$ replacing $\eps$. Thus if $A \subset \Z/q\Z \times [0,1]$ is open, $\mu(A) \geq \frac{1}{3} + \frac{1}{8}\eps$ and $T(A) \leq \eta'$ then $\mu(A \cap (H \times [0,\eps'])) \ll_{\eps} \eta'$, where $H \leq \Z/Q\Z$ is a subgroup of index at most $C_{\eps}$ and $\eps'\gg_\eps 1$. Let $M = C_{\eps}!$. Then the index $[\Z/q\Z : H]$ necessarily divides $M$, so for every $x \in \Z/q\Z$ we have $Mx \in H$.

For $t \in \left[\frac{1}{2},1\right]$ and $q\in\N$ define $\pi_t : \Z/q\Z \times [0,1] \rightarrow \Z/Mq\Z \times [0,1]$ by $\pi_t(x,y) = (Mx, t\eps' y)$. Then, by the above, if $A \subset \Z/Mq\Z \times [0,1]$ is open, $\mu(A) \geq \frac{1}{3} + \frac{1}{8}\eps$ and $T(A) \leq \eta'$ then $\mu(A \cap \im\pi_t)\ll_{\eps} \eta'$.

Let $Q' = MQ$ and define $w_t'$ on $\Z/Q'\Z\times[0,1]$ by
\begin{equation}\label{legendary-wt}
  w_t'(x) = \tfrac{3}{4} 1_{\im\pi_t}(x) \frac{w(\pi_t^{-1}(x))}{\mu(\im\pi_t)} + \tfrac{1}{4}.
\end{equation}
This definition can be made to look a little more natual as follows. Define measures $\nu$ on $\Z/Q\Z\times[0,1]$ and $\nu_t'$ on $\Z/Q'\Z \times [0,1]$ by $\nu(A) = \int 1_A w\,d\mu$ and $\nu_t'(A) = \int 1_A w_t'\,d\mu$. Then the relationship between $\nu$ and $\nu'_t$ is $\nu_t' = \frac{3}{4} (\pi_t)_* \nu + \frac{1}{4} \mu$, where the push-forward measure is defined as usual by $\pi_* \nu(A) = \nu(\pi^{-1} (A))$.

Now suppose $Q'\mid q'$, say $q' = Mq$ where $Q\mid q$, and $A\subset\Z/q'\Z\times[0,1]$ is open and $T(A)\leq\eta'$. If $\mu(A)\geq\frac{1}{3}+\frac{1}{8}\eps$ then, as noted above, $\mu(A \cap \im \pi_t) \ll_{\eps} \eta'$. Therefore \[ (\pi_t)_* \nu(A) = \nu(\pi_t^{-1}(A))  \leq \|w\|_{\infty} \mu(\pi_t^{-1}(A))  = \|w\|_\infty \frac{\mu(A \cap \im \pi_t)}{\mu(\im \pi_t)} \ll_{\eps} \eta' \| w \|_{\infty},\] so $\nu'_t(A) \leq O_{\eps}(\eta' \|w\|_\infty) + \frac{1}{4}$, and this can be made to be less than $\frac{1}{3}$ by taking $\eta'$ sufficiently small depending on $\eps$ and $\|w\|_\infty$.

Suppose instead $\mu(A)\leq\frac{1}{3}+\frac{1}{8}\eps$. If $\eta'\leq\mu(\im \pi_{1/2})^3\eta$ then we have $T(\pi_t^{-1}(A)) \leq \eta$ for all $t \in [\frac{1}{2}, 1]$, so in this case we have \[ \nu'_t(A) = \tfrac{3}{4} \nu(\pi_t^{-1}(A)) + \tfrac{1}{4}\mu(A) \leq \tfrac{3}{4}\alpha + \tfrac{1}{4}\left(\tfrac{1}{3} + \tfrac{1}{8}\eps\right) = \alpha'.\]

To complete the proof, we must show that $w'$ can be chosen to be Lipschitz. In fact $w_t'$ generally has a jump discontinuity\footnote{This would not, in actual fact, be a fatal hole in our argument; in Section \ref{sec3} we could instead have dealt with the larger class of piecewise Lipschitz functions. This is a little complicated, however, and the Lipschitz hypothesis is convenient.} at every point of the form $(Mk,t\eps')$, but we can remedy this by defining \begin{equation}\label{w-smooth} w' = 2\int_\frac{1}{2}^1 w'_t\,dt.\end{equation} Then if $Q'\mid q'$, $A\subset\Z/q'\Z\times[0,1]$ is open and $T(A)\leq\eta'$ then \[\int_A w'\,d\mu = 2\int_\frac{1}{2}^1 \nu'_t(A)\,dt \leq 2\int_\frac{1}{2}^1 \alpha'\,dt = \alpha',\]
while from \eqref{legendary-wt} and \eqref{w-smooth} it is fairly clear that $w'$ is Lipschitz.
\end{proof}

By applying Lemma \ref{weight-lem} iteratively we get a weak version of Proposition \ref{nu-exist-3}.

\begin{proposition}\label{nu-exist-weak}
Let $\eps > 0$. Then there is an integer $Q$ and a Lipschitz weight function $w : \Z/Q\Z \times [0,1] \to (0,\infty)$ with the following property. If $Q\mid q$, then for any open set $A\subset\Z/q\Z\times[0,1]$ such that $\int 1_A w\,d\mu\geq\tfrac{1}{3}+\eps$ we have $T(A)\gg_\eps 1$.
\end{proposition}
\begin{proof} Apply Lemma \ref{weight-lem} iteratively starting with $Q=1$, $w \equiv 1$, $\alpha = 1$, and $\eta = 1$. After $n = 100\log(1/\eps)$ steps we obtain an integer $Q$ and a weight $w$ on $\Z/Q\Z \times [0,1]$ satisfying the hypotheses of that lemma with some $\eta>0$ and \[ \alpha = \left(\tfrac{3}{4}\right)^n + \tfrac{1}{4}\left(1 + \tfrac{3}{4} + \left(\tfrac{3}{4}\right)^2 + \dots + \left(\tfrac{3}{4}\right)^{n-1} \right) \left(\tfrac{1}{3} + \tfrac{1}{8}\eps\right) < \tfrac{1}{3} + \tfrac{1}{4}\eps.\qedhere\]
\end{proof}

To obtain Proposition \ref{nu-exist-3} from Proposition \ref{nu-exist-weak}, we must replace $1_A$ by an arbitrary continuous function $\Psi$, and we must introduce the additional factor of $\T = (\R/\Z)^d$. Both of these improvements turn out to be relatively straightforward.

\begin{proof}[Proof of Proposition \ref{nu-exist-3}]
We will show that $w$, the weight function on $\Z/Q\Z\times[0,1]$ constructed in Proposition \ref{nu-exist-weak}, has property required by Proposition \ref{nu-exist-3}. We do this in stages, beginning with the following.

\emph{Claim I}. Consider the ``discrete torus'' $\T_{\disc} = \Z/q_1 \Z \times \dots \times \Z/q_d \Z$, where $q_1,\dots, q_d > q$ are distinct primes. Suppose $A \subset \Z/q\Z \times [0,1] \times \T_{\disc}$ is open and $\int 1_A (w \times 1_{\T_{\disc}})d\mu \geq \frac{1}{3} + \frac{1}{4}\eps$. Then $T(A) \gg_{\eps} 1$. (Here, $\Z/Q\Z \times [0,1] \times \T_{\disc}$ is endowed with the uniform probability measure $\mu$.)

\emph{Proof of Claim I}. Let $q' = q q_1 \dots q_d$ and consider the $\mu$-preserving isomorphism \[\Z/q'\Z \times [0,1] \stackrel{\psi}{\longrightarrow} \Z/q\Z \times [0,1]  \times \T_{\disc}\] given by \[ \psi(x,y) = (x \md q, y, x\md{q_1}, \dots, x \md{q_d}).\] If $A\subset\Z/q\Z\times[0,1]\times\T_{\disc}$ is open and $\int 1_A (w \times 1_{\T_{\disc}})\,d\mu \geq \frac{1}{3} + \frac{1}{4}\eps$ then $\int 1_{\psi^{-1}(A)} w\,d\mu \geq \frac{1}{3} + \frac{1}{4}\eps$, so Proposition \ref{nu-exist-weak} implies $T(A) = T(\psi^{-1}(A))\gg_{\eps} 1$.

\emph{Claim II}. The same claim holds if $\T_{\disc}$ is replaced by the genuine torus $\T = (\R/\Z)^d$ and $\frac{1}{4}\eps$ is replaced with $\frac{1}{2}\eps$. That is, if $A \subset \Z/Q\Z\times [0,1] \times \T$ is open and $\int 1_A (w \times 1_{\T}) \geq \frac{1}{3} + \frac{1}{2}\eps$ then $T(A) \gg_{\eps} 1$.

\emph{Proof of Claim II}. This is a standard discretisation argument. We may find some $\kappa>0$ and a subset $A'\subset A$ such that $A'$ is a disjoint union of sets of the form $\{x\}\times I\times J$, where $I\subset[0,1]$ is an open interval of length $\kappa$ and $J\subset\T$ is an open cube of side-length $\kappa$, and $\int 1_{A'} (w\times 1_{\T}) \geq \frac{1}{3} + \frac{1}{3}\eps$. Regard $\T_{\disc}$ as a subgroup of $\T$ by mapping $(x_1,\dots, x_d) \in \T_{\disc}$ to $\left(\frac{x_1}{q_1}, \dots, \frac{x_d}{q_d}\right) \in \T$. Set \[A'_{\disc} = A' \cap (\Z/Q\Z\times[0,1]\times\T_{\disc}).\] Then as $q_1,\dots,q_d\to\infty$ both $\int 1_{A'_{\disc}}(w\times1_{\T_{\disc}})\,d\mu \to\int 1_{A'} (w\times1_{\T})\,d\mu$ and $T(A'_{\disc})\to T(A')$, so by the previous claim $T(A)\geq T(A')\gg_\eps 1$.

Finally, we are ready to verify Proposition \ref{nu-exist-3} itself, which differs from Claim II only in the presence of a general continuous function $\Psi : \Z/q\Z \times [0,1] \times \T \rightarrow [0,1]$ in place of the characteristic function $1_A$. Suppose that $\Psi$ is given and $\int \Psi\cdot(w\times 1_{\T})\,d\mu > \frac{1}{3} + \eps$. Consider the open set $A = \left\{x: \Psi(x)>\frac{1}{2}\eps\right\} \subset\Z/q\Z\times[0,1]\times \T$. Since \[\tfrac{1}{3} + \eps  \leq \int \Psi \cdot (w\times 1_{\T}) = \int_A \Psi\cdot(w\times 1_{\T}) + \int_{A^c} \Psi\cdot(w\times 1_{\T}) \leq \int 1_A (w\times 1_{\T}) + \tfrac{1}{2}\eps,\] we have $\int 1_A (w\times 1_{\T}) \geq \frac{1}{3} + \frac{1}{2}\eps$. By Claim II we therefore have $T(A) \gg_{\eps} 1$, and thus $T(\Psi) \geq \left(\frac{1}{2}\eps\right)^3 T(A) \gg_{\eps} 1$. This (at last!) completes the proof of Proposition \ref{nu-exist-3} and hence of Theorem \ref{mainthm}.
\end{proof}

\section{More on sets of doubling less than $4$}
\label{sec6}

The purpose of this section is to expand just a little more on the results of Section \ref{sec4}, which may be of independent interest. The first theorem below is a direct consequence of Theorem \ref{discthm}; the second is the corresponding consequence of Corollary \ref{thm:stronger-thmconj} in the case $Q = 1$.

\begin{theorem}\label{int-thm}
If $A \subset \{1,\dots,N\}$ is a set such that $|A - A| \leq 4|A| - \eps N$ then there is an arithmetic progression $P \subset \{1,\dots,N\}$ of length $\gg_{\eps} N$ on which $A$ has density at least $\frac{1}{2} + \tfrac{1}{5}\eps$.
\end{theorem}

\begin{theorem}\label{real-thm}
If $A \subset [0,1]$ is an open set such that $|A - A| \leq 4|A| - \eps$ then there is an interval $I \subset [0,1]$ of length $\gg_{\eps} 1$ on which $A$ has density at least $\frac{1}{2} + \tfrac{1}{7}\eps$.
\end{theorem}

\emph{Remarks.}
\begin{enumerate}
	\item Neither the constant $\tfrac{1}{5}$ in Theorem \ref{int-thm} (or Theorem \ref{discthm}) nor the constant $\tfrac{1}{7}$ in Theorem \ref{real-thm} is optimal. Indeed, if one allows the implied constants to depend on $\eta$ then our proof can be modified to get $\tfrac{1}{4}-\eta$ for both these constants.
	\item A similar conclusion to Theorem \ref{real-thm} could be obtained if one instead had \emph{two} sets $A$ and $B$ satisfying $|A - B| < |A| + |B| + 2|A|^{1/2} |B|^{1/2} - \eps$. The conclusion would then be that there are intervals $I_A$ and $I_B$ such that the densities of $A, B$ on $I_A, I_B$ respectively sum to at least $1 + \tfrac{1}{3}\eps$ (or up to $1 + (\tfrac{1}{2}-\eta)\eps$, constants depending on $\eta$). There would be a similar generalisation of Theorem \ref{int-thm}. We leave the proof of these results as an exercise to the interested reader. One annoying additional complication would be the need to have an arithmetic regularity lemma valid for two sets simultaneously. While such a statement can be easily established by modifying the arguments of \cite{green-tao-arithregularity}, no such result currently appears in the literature.
	\item We are not aware of any reason that the length of $I$ or $P$ could not be bounded below by some quite reasonable function of $\eps$, but our argument does not give one. Much better quantitative results are available under the assumption that $|A - A| < 3|A|$: see \cite{ruzsa-measure}.
\end{enumerate}

Consider the discrete case $A\subset \{1,\dots,N\}$. Note that the hypothesis $|A-A|\leq 4|A| - \eps N$ implies $|A|\geq \tfrac{1}{4}\eps N$. Thus one can consider Theorem \ref{int-thm} to contain a ``hidden hypothesis'' to the effect that $A$ is somewhat dense in $\{1,\dots,N\}$. If instead one assumed only that $|A-A|\leq (4-\eps)|A|$, then our argument would give bounds depending on $\alpha = |A|/N$ as well as $\eps$. Using a ``Freiman modelling'' argument of a type pioneered by Ruzsa, however, we can overcome this. We first isolate a lemma due to Lev \cite{lev} (though earlier results of S\'ark\"ozy \cite{sarkozy} would also suffice). 

\begin{lemma}
Let $P$ be a finite arithmetic progression of even length greater than $12$, and let $X \subset P$ be a set with $|X| > \frac{1}{2}|P|$. Then $5X - 4X$ contains $P$. 
\end{lemma}
\begin{proof}
The statement of the lemma being affine-invariant, we may suppose without loss of generality that $P = \{1,\dots,N\}$. Since $|X| > N/2$, the highest common factor of the elements of $X$ is 1. By \cite[Lemma 1]{lev} with $k = 2$ (and a short computation), $4X$ contains an interval of length at least $4(\frac{N}{2} - 3) > N$. It follows that $4X - 4X$ contains $\{-N,\dots,N\}$, and the result follows immediately.\end{proof}

In the proof of the next lemma we will use the notion of a \emph{Freiman homomorphism}. See \cite[Definition 5.21]{tv} for details.

\begin{theorem}\label{dub-4-thm}
Let $\eps > 0$. Suppose that $A$ is a finite set of integers such that $|A - A| \leq (4 - \eps)|A|$. Then there is an arithmetic progression $P \subset \Z$ of length $\gg_{\eps} |A|$ on which the density of $A$ is at least $\frac{1}{2} + c\eps$.
\end{theorem}
\begin{proof}
By \cite[Theorem 1.4]{green-ruzsa-rectification}, every set $A \subset \Z$ with $|A - A| \leq 4|A|$ is Freiman $18$-isomorphic to a subset of $\{1,\dots,N\}$ for some $N \ll |A|$. Let $\pi : A \to A' \subset \{1,\dots,N\}$ be this Freiman isomorphism. Then clearly $|A'|/N \gg 1$, so $|A' - A'| \leq 4|A'| - \eps' N$ with $\eps' \gg \eps$. By Theorem \ref{int-thm} to $A'$, there is a progression $P' \subset \{1,\dots,N\}$ of length $|P'| \gg_{\eps} N$ on which the density of $A'$ is at least $\frac{1}{2} + \tfrac{1}{5}\eps'$. Finally, by the preceding lemma, $P' \subset 5A' - 4A'$, so it follows from basic facts about Freiman homomorphisms (see \cite[Section 5.2]{tv} for example) that $\pi^{-1} : A' \to A$ induces a Freiman $2$-homomorphism $\tilde\pi^{-1} : P' \to \Z$ coinciding with $\pi^{-1}$ on $A' \cap P'$. The image $P = \tilde\pi^{-1}(P')$ is then a progression of length $\gg_\eps |A|$ on which $A$ has density at least $\frac{1}{2} + c \eps$.
\end{proof}

\emph{Remarks.} 
\begin{enumerate}
	\item The value for $c$ given by this argument is something like $2^{-1000}$.
	\item Under stronger conditions such as $|A + A| < 3|A|$ or $|A - A| < 3|A|$, more precise information can be obtained: see \cite[Theorem 1.9]{freiman-book} or \cite{lev-smeliansky}.
	\item Statements of the same form as Theorem \ref{dub-4-thm}, but with $\frac{1}{2} + c\eps$ replaced by some small quantity $f(\eps) > 0$, follow from versions of Freiman's theorem \cite[Theorem 2.8]{freiman-book}, \cite[Theorem 1.2]{bilu}. These statements come with more effective lower bounds on the length of $P$.
\end{enumerate}

\appendix

\section{Regularity and counting lemmata}

\label{appA}

In this appendix we collect some tools used in the main part of the paper. All of these are more or less standard, or at least have easily quotable references.

\emph{The arithmetic regularity lemma.} We begin with the arithmetic regularity lemma, the main result of \cite{green-tao-arithregularity}, used twice in the paper. As reassurance to the reader who views that paper with trepidation, we remark that the majority of it is given to applications, and only Sections 1 and 2 are relevant to us. Furthermore, that paper establishes a regularity lemma for the Gowers $U^{s+1}$-norm for general $s$, whereas we only need the case $s = 1$. This means that the notion of a \emph{nilsequence}, beyond the abelian case, is not relevant here. A complete, self-contained proof of the arithmetic regularity lemma in the form we need it here can be written up in less than 10 pages. The first-named author has provided such a write-up online \cite{sean-writeup}.

We begin by defining a quantitative notion of irrationality for vectors $\theta \in \R^d$.

\begin{definition}\label{irrational-def}
Suppose that $\theta \in \R^d$. Let $N \geq 1$ be an integer and let $A > 0$ be some real parameter. We say that $\theta$ is $(A, N)$-irrational if whenever $q_1, \dots, q_d$ are integers, not all zero, with $\sum_i |q_i| \leq A$ we have $\| q_1 \theta_1 + \dots + q_d \theta_d \|_{\R/\Z} \geq A/N$.
\end{definition}

\begin{lemma}\label{ar} Suppose we are given a parameter $\delta > 0$ and a growth function $\mathcal{F} : \N \rightarrow \R_+$. Then for any function $f : \{1,\dots, N\} \to [0,1]$ there is an $M \ll_{\delta,\mathcal{F}} 1$ and a decomposition $f = f_{\tor} + f_{\sml}+ f_{\unf}$ into functions taking values in $[-1,1]$,  where $\| f_{\sml} \|_{\ell_2(N)} \leq \delta$, $\| f_{\unf} \|_{U^2(N)} \leq 1/\mathcal{F}(M)$ and $f_{\tor}(n) = F(n \mdlem{q}, n/N, \theta n)$ for some $q,d \leq M$ and some function $F : \Z/q\Z \times [0,1] \times (\R/\Z)^d \rightarrow [0,1]$ with Lipschitz constant at most $M$. Furthermore the element $\theta \in (\R/\Z)^d$ may be taken to be $(\mathcal{F}(M), N)$-irrational.
\end{lemma}

Here $\| g \|_{U^2(N)}$ is the Gowers $U^2(N)$-norm, whose definition will be recalled below. We do not offer a proof of this lemma, but merely a guide to extracting this result from \cite{green-tao-arithregularity}. The function $f_{\tor}$ written here is the same thing as, in the language of that paper, a ``$(\mathcal{F}(M), N)$-irrational virtual nilsequence of degree $\leq 1$, complexity $\leq M$ and scale $N$''. Once we have justified this assertion, Lemma \ref{ar} is essentially the same as \cite[Theorem 1.2]{green-tao-arithregularity}, the proof of which occupies Section 2 of that paper. The definition of an irrational virtual nilsequence of degree $\leq s$ is rather long and complicated for general $s$, but for $s = 1$ a great deal simplifies: a \emph{filtered nilmanifold} of degree $1$ and complexity $\leq M$ (cf. \cite[Definition 1.4]{green-tao-arithregularity}) is just the torus $(\R/\Z)^d$ for $d \leq M$, a \emph{polynomial orbit} of degree $1$(cf. \cite[Definition 1.7]{green-tao-arithregularity}) is just a sequence of the form $n \mapsto \theta n$ for some $\theta \in (\R/\Z)^d$ and a \emph{virtual nilsequence} of degree $1$ and complexity $\leq M$ at scale $N$ (cf. \cite[Definition 1.9]{green-tao-arithregularity}) is just a function $F(n \md{q}, n/N, \theta n)$ with $\theta \in (\R/\Z)^d$, $d,q,\|F\|_{\Lip} \leq M$. Finally, an $(A,N)$-irrational sequence (cf. \cite[Definition A.6]{green-tao-arithregularity}), in the case that the sequence has the form $n\mapsto\theta n$ where $\theta \in (\R/\Z)^d$, coincides with the notion of irrationality defined above.

\emph{Equidistribution and counting.} If $\theta \in (\R/\Z)^d$ is highly irrational in the sense of Definition \ref{irrational-def}, then the sequence $\theta n$ is highly equidistibuted on $(\R/\Z)^d$ as $n$ ranges over fairly long progressions. Moreover, the triple $(n \md{q}, n/N, n\theta)$ is highly equidistributed in $\Z/q\Z \times [0,1] \times n\theta$ as $n$ varies over $\{1,\dots,N\}$. We prove various statements of this type required in the main body of the paper. The first is quite classical.

\begin{lemma}\label{distribution-integral-a}
Suppose that $\theta \in (\R/\Z)^d$ is $(A, N)$-irrational, and let $F : (\R/\Z)^d \rightarrow \C$ be a function with Lipschitz constant at most $M$. Suppose that $P \subset \{1,\dots, N\}$ is an arithmetic progression of length at least $\eta N$. Then, provided that $A > A_0(M, d, \eta, \delta)$ is large enough,
\[ \left|\E_{n \in P} F(\theta n) - \int F \, d\mu \right| \leq \delta.\]
\end{lemma}
\begin{proof}
The key here (as usual in equidistribution theory) is to take a Fourier expansion of $F$ and truncate it. In particular, we may find $M_0 = O_{M, d, \delta}(1)$ and coefficients $c_m$ with $c_0 = \int F$ and $c_m = O_{M,d}(1)$ for $m\neq 0$ such that
\[  \left| F(x) - \sum_{|m| \leq M_0} c_m e(m \cdot x) \right| \leq \delta/2 \]
uniformly in $x$. For a proof, see for example \cite[Lemma A.9]{green-tao-quadraticuniformity}. 
It follows, of course, that 
\[
  \left| \E_{n \in P} F(\theta n) - \int F\, d\mu \right| \leq  \sum_{|m| \leq M_0, m \neq 0} |c_m| | \E_{n \in P} e(m \cdot \theta n)| + \frac{\delta}{2}.
\]
Thus we need only show that 
\[ \E_{n \in P} e(m \cdot \theta n) = o_{m, \eta; A \rightarrow \infty}(1),\] and then take $A$ sufficiently large. If the common difference of the arithmetic progression $P$ is $h$, then by summing the geometric progression we have the bound
\[ \E_{n \in P} e(m \cdot \theta n) \ll \frac{1}{\eta N \| (m \cdot \theta) h \|_{\R/\Z}}.\] If $A > |h|(|m_1| + \dots + |m_d|)$ (where $m = (m_1,\dots, m_d)$) then, by the definition of $(A,N)$-irrationality, $\| (m \cdot \theta) h \|_{\R/\Z} \geq A/N$. The result follows immediately. 
\end{proof}

Our second result is a little more involved, but is proved in essentially the same way as the last lemma. It states that if $\theta$ is highly irrational and $N$ is sufficiently large then $(n \md{q}, n/N, \theta n)$ is highly equidistributed in $\Z/q\Z \times [0,1] \times (\R/\Z)^d$. 

\begin{lemma}\label{distribution-integral}
Suppose that $\theta \in (\R/\Z)^d$ is $(A, N)$-irrational. Let $q \in \N$, and let $F : \Z/q\Z \times [0,1] \times (\R/\Z)^d \rightarrow \C$ be a function with Lipschitz constant at most $M$. Let $\delta > 0$ be arbitrary. Then, provided that $A > A_0(M, q, d, \delta)$ and $N>N_0(M,q,d,\delta)$ are large enough,
\[ 
  \left| \E_{n \leq N} F(n \mdlem{q}, n/N, \theta n) - \int F \, d\mu \right| \leq \delta.
\]
\end{lemma}
\begin{proof}[Proof sketch]
Again the idea is to take a truncated Fourier expansion of $F$, but because $F|_{\Z/q\Z\times\{0\}\times(\R/\Z)^d}$ and $F|_{\Z/q\Z\times\{1\}\times(\R/\Z)^d}$ need not agree the expansion looks a little more complicated. The key point is that $F$ can be extended to an $M$-Lipschitz function $\Z/q\Z\times[-1,1]\times(\R/\Z)^d\to\C$ such that $F(x,y,z)=F(x,-y,z)$, so $F$ may be approximated by a sum of the functions $\phi_{a, m, \mathbf{m}}$ given by
\begin{equation}\label{phi}\phi_{a,m ,\mathbf{m}}(x,y,z) = e\left(\frac{a}{q}x+ \frac{m}{2}y + \mathbf{m} \cdot z\right) + e\left(\frac{a}{q}x - \frac{m}{2}y + \mathbf{m} \cdot z\right),\end{equation}
where $a \in \Z/q\Z, m \in \Z$ and $\mathbf{m} \in \Z^d$. Then just as in the proof of the previous lemma we need only check that 
\begin{equation}\label{to-check-22} \E_{n \leq N} \phi_{a, m, \mathbf{m}} (n \md{q}, n/N, \theta n) = o_{a, m, \mathbf{m}, q; A,N \to \infty}(1)\end{equation}
provided that $a,m,\mathbf{m}$ are not all zero. Substituting in, the left-hand side is 
\begin{equation}\label{to-check-3}   \E_{n \leq N} \left(e\left(\left(\frac{a}{q} + \frac{m}{2N} + \mathbf{m} \cdot \theta\right) n  \right) + e\left(\left(\frac{a}{q} - \frac{m}{2N} + \mathbf{m} \cdot \theta\right) n  \right)\right). \end{equation}
Summing the geometric progressions, we see that this is bounded by $\eps$ unless
\begin{equation}\label{to-use-9} \left\| \frac{a}{q} + \frac{m}{2N} + \mathbf{m} \cdot \theta \right\|_{\R/\Z} \ll \frac{1}{N\eps}. \end{equation}

Supposing first that $\mathbf{m} \neq 0$, inequality \eqref{to-use-9} implies
\[   \left\| \frac{mq}{2N} + q \mathbf{m} \cdot \theta \right\|_{\R/\Z} \ll \frac{q}{N\eps}, \]
and hence
\[ \left \| \mathbf{m'} \cdot \theta \right \|_{\R/\Z} \ll \frac{q}{N\eps} + \frac{mq}{2N}, \]
where $\mathbf{m'} = q\mathbf{m}$. If $A$ is sufficiently large in terms of $\eps, q, m$ and $\mathbf{m}$, this is contrary to the $(A,N)$-irrationality of $\theta$.

Now suppose that $\mathbf{m} = 0$. Then if $N$ is large enough depending on $m$ and $q$, \eqref{to-use-9} implies that $a = 0$. Thus $a = \mathbf{m} = 0$, and hence $m \neq 0$. Then the expression \eqref{to-check-3} is $\E_{n \leq N}\left(e(mn/2N) + e(-mn/2N)\right) = 0$, so \eqref{to-check-22} certainly follows in this case as well.
\end{proof}

%A corollary of this is the following result.
%
%\begin{lemma}\label{cube-distribution} Suppose that $\theta \in (\R/\Z)^d$ is $(A, N)$-irrational. Let $q \in \N$, and let $\Sigma \subset \Z/q\Z \times [0,1] \times (\R/\Z)^d$ be a ``cube'' of the form $\{x\} \times [t, t+\eta] \times \prod_{i = 1}^d [t_1, t_i + \eta]$. If $A > A_0(q, d, \eta)$ and $N>N_0(q,d,\eta)$ are sufficiently large then
%\[ \frac{2}{3} \mu(\Sigma) \leq \frac{1}{N} |\{ n \in \{1,\dots, N\} : (n \mdlem{q}, n/N, \theta n) \in \Sigma \} | \leq 2 \mu(\Sigma).\]
%\end{lemma}
%\begin{proof}
%If the characteristic function $1_{\Sigma}$ were Lipschitz, this would be nothing more than a special case of Lemma \ref{distribution-integral}. Of course it is not, but for every $\eps$ there are ``tent'' functions $\chi_{-}, \chi_+$ with $\chi_-(x) \leq 1_{\Sigma}(x) \leq \chi_+(x)$ for all $x$, $\int \chi_+ \leq \int \chi_- + \eps$ and $\| \chi_- \|_{\Lip}, \| \chi_+ \|_{\Lip} = O_{\eps, \eta, d}(1)$. Taking $\eps$ sufficiently small, the result does now follow from Lemma \ref{distribution-integral}. 
%\end{proof}

The next result is a kind of ``counting lemma''. It eventually allows one to relate summing triples $(x, y, x+y)$ in $A \subset \{1,\dots,N\}$ with summing triples in $(\R/\Z)^d$ weighted by $f_{\tor}$.

\begin{lemma}\label{counting-lemma}
Suppose that $\theta \in (\R/\Z)^d$ is $(A, N)$-irrational. Let $q \in \N$, and let $F : \Z/q\Z \times [0,1] \times (\R/\Z)^d \to \C$ be a function with Lipschitz constant at most $M$. Let $f(n) = F(n \mdlem{q}, n/N, \theta n)$. If $A > A_0(q,M,\eps)$ and $N>N_0(q,M,\eps)$ are large enough then 
\[ |T(f) - T(F)| = |T_{\{1,\dots, N\}}(f) - T_{\Z/q\Z \times [0,1] \times (\R/\Z)^d}(F)| \leq \eps.\]
\end{lemma}
\begin{proof}[Proof sketch] Write $\pi(n) = (n \md{q}, n/N, \theta n)$. We showed in Lemma \ref{distribution-integral} that, if $\theta$ is highly irrational, $\pi(n)$ is highly equidistributed in $X = \Z/q\Z \times [0,1] \times (\R/\Z)^d$ as $n$ ranges in $\{1,\dots,N\}$. It follows that as $n,n'$ range over $\{1,\dots,N\}$, the pair $(\pi(n), \pi(n'))$ is highly equidistributed in $X \times X$, and in particular \[\E_{n, n' \in \{1,\dots, N\}} F_*(\pi(n), \pi(n')) \approx \int F_*(x,x') d\mu(x) d\mu(x')\] for any Lipschitz function $F_* : X \times X \rightarrow \C$. Applying this with $F_*(x,x') = F(x) F(x') F(x + x')$ gives the stated result. 
\end{proof}

Finally, we require the following simple result which does not mention $\theta$ at all.

\begin{lemma}\label{lem-a7} Let $q \in \N$. Suppose that $w : \Z/q\Z \times [0,1] \rightarrow \C$ has Lipschitz constant at most $M$. Then 
\[ \E_{n \leq N} w(n \mdlem{q}, n/N) = \int w\, d\mu + o_{q,M;N \rightarrow \infty}(1).\]
\end{lemma}
\begin{proof}[Proof sketch] Split into progressions $P_a = \{n \leq N: n \equiv a \md{q}\}$. Then one need only show that
\[  \E_{n \in P_a} w(a, n/N) = \int_0^1 w(a,y) dy + o_{M; N \rightarrow \infty}(1), \]
which is fairly obvious from the definition of the Riemann integral.
\end{proof}

\emph{Properties of the Gowers $U^2$-norm.} The statement of the arithmetic regularity lemma involved the Gowers $U^2(N)$-norm of a function. Here we recall some basic properties of this norm, whose proofs may be found in several places. We begin by recalling the definition. For a fuller discussion, see \cite{green-tao-arithregularity}. 

\begin{definition}
Let $f : \{1,\dots, N\} \rightarrow \C$ be a function. Then we define $\| f \|_{U^2(N)} = \| f \|_{U^2(G)}/\| 1_{\{1,\dots,N\}} \|_{U^2(G)}$, where $G = \Z/N'\Z$ for some arbitrary $N' > 4N$ and
\[ \| f \|_{U^2(G)}^4 = \E_{x, h_1, h_2 \in G} f(x) \overline{f(x+h_1) f(x+h_2)} f(x + h_1 + h_2).\]
\end{definition}
In this definition, $f$ is regarded (by abuse of notation) as a function on $G$ by defining $f(x) = 0$ if $x \in G \setminus \{1,\dots, N\}$, where $\{1,\dots, N\}$ is regarded as embedded in $G$ in the natural manner. It is not hard to see that this definition does not depend on the exact choice of $N'$. Introducing the group $G$ is a technical device which is useful in several parts of the theory. 

We begin with a standard lemma. 

\begin{lemma}\label{gowers-progressions}
Suppose that $f : \{1,\dots,N\} \rightarrow \C$ is a function. Then $|\E_{n \leq N} f(n)| \ll \| f \|_{U^2(N)}$.
More generally suppose that $P \subset \{1,\dots,N\}$ is a progression of length at least $\eta N$. Then $|\E_{n \in P} f(n)| \ll \eta^{-1} \| f \|_{U^2(N)}$.
\end{lemma}
\begin{proof}  We establish the second statement, the first being a special case of it. Fix a prime $N' \in [4N,8N]$ and write $G = \Z/N'\Z$ as in the definition of the Gowers $U^2(N)$-norm. Note that the $U^2(N)$-norm and the $U^2(G)$-norm are comparable up to an absolute constant. We use the inequality
\begin{align*} |\E_{x \in G} f(x) g(x)| 
  & = \left| \sum_r \hat{f}(r) \hat{g}(r) \right| \\
  & \leq \left(\sum_r |\hat{f}(r)|^4 \right)^{1/4} \left(\sum_r |\hat{g}(r)|^{4/3} \right)^{3/4}  \\ 
  & = \| f \|_{U^2(G)} \left(\sum_r |\hat{g}(r)|^{4/3} \right)^{3/4}.
\end{align*}
Here the Fourier transform $\hat{f}(r) = \E_{x \in G} f(x) e(-rx/N')$ is the discrete Fourier transform on $G$, and we have used H\"older's inequality and the well-known fact (see, for example, \cite[Chapter 11]{tv}) that $\| f \|_{U^2(G)} = \| \hat{f} \|_{\ell^4}$. Taking $g = 1_P$, the characteristic function of the progression $P$, it suffices to show that $\sum_r |\hat{g}(r)|^{4/3} = O(1)$. Dilating, we may assume that the common difference of $P$ is $1$. But then we have, upon summing the geometric progression, the bound $|\hat{g}(r)| \ll \min(1, |r|^{-1})$, from which the result follows immediately.\end{proof}

The next lemma is a more complicated result along similar lines.

\begin{lemma}\label{gowers-orthog-struct}
Suppose that $d, q, M \in \N$ and that $\delta > 0$. There for some $\delta_*=\delta_*(d, q, M,\delta)>0$ and all sufficiently large $N \geq N_0(d,q,M,\delta)$ the following is true. Let $f(n) = F(n\md q, n/N,\theta n)$, where $F:\Z/q\Z \times [0, 1] \times (\R / \Z)^d\to[0,1]$ is $M$-Lipschitz and $\theta \in (\R / \Z)^d$, and suppose $g : \{1,\dots,N\} \to [-1,1]$ satisfies $\|g\|_{U^2(N)} \leq \delta_*$.  Then $\left| \E_{n \in N} f(n) g(n) \right| \leq \delta$.
\end{lemma}

In other words, the ``structured objects'' and the ``pseudorandom objects'' of the regularity lemma do not correlate.

\begin{proof}[Proof sketch]
As in the proof of Lemma \ref{distribution-integral}, the idea is to decompose $F$ as a Fourier expansion of length $O_{\delta, d, q, M}(1)$, plus a uniformly small error:
\[ F = \sum_{a, m, \mathbf{m}} c_{a, m, \mathbf{m}} \phi_{a, m, \mathbf{m}}  + F_{\sml}, \]
where $\phi_{a,m,\mathbf{m}}$ is given by \eqref{phi}, $|c_{a, m, \mathbf{m}}| \leq 1$ and $\|F_{\sml}\|_{\ell^\infty} \leq \tfrac{1}{2}\delta$. Note
\begin{equation}\label{eq77}
  \phi_{a,m,\mathbf{m}}(n \md{q}, n/N, \theta n) = e(\beta_+ n) + e(\beta_- n),
\end{equation}
where $\beta_\pm = \frac{r}{q} \pm \frac{m}{2N} + \mathbf{m} \cdot \theta$ (though this exact form is unimportant). However, writing $e(\beta n) = e(\beta (n + h_1 + h_2)) e(-\beta h_1) e(-\beta h_2)$ and averaging over $h_1, h_2$ it follows by the Gowers-Cauchy-Schwarz inequality \cite[Chapter 11]{tv} that $|\E_{n \leq N} e(\beta n) g(n)| \ll \| g \|_{U^2(N)}$ uniformly in $\beta \in \R$, so
\[ \left|\E_{n\leq N} \phi_{a,m,\mathbf{m}}(n\md q, n/N,\theta n) g(n)\right| \ll \delta_* \]
uniformly in $a,m,\mathbf{m}$. It follows that
\begin{align*}
	\left| \E_{n \leq N} f(n) g(n) \right|
	&= \left| \E_{n \in N} F(n\, \md{q}, n / N, \theta n) g(n) \right| \\
	&\leq \sum_{a,m,\mathbf{m}} \left|\E_{n\leq N} \phi_{a,m,\mathbf{m}}(n\md q,n/N,\theta n) g(n)\right| + \tfrac{1}{2}\delta\\
  &\leq O_{\delta, d, q, M}(\delta_*) + \delta / 2,
\end{align*}
so for $\delta_*$ sufficiently small depending on $\delta$, $d$, $q$ and $M$, $|\E_{n\leq N}f(n)g(n)|\leq\delta$.
\end{proof}

\emph{Miscellany.} We turn now to some rather miscellaneous lemmas. Recall that if $f : \{1,\dots,N\}\to\C$ is a function then $T(f) = \E_{n,n' \leq N} f(n) f(n') f(n + n')$.

\begin{lemma}\label{lem-a9} Suppose that $f, \tilde f : \{1,\dots, N\} \to [-1,1]$ are functions. Then $|T(f)-T(\tilde{f})| \leq 7\|f-\tilde{f}\|_{\ell^1(N)}$ and $|T(f) - T(\tilde{f})| \ll \| f - \tilde{f} \|_{U^2(N)}$.\end{lemma}

Recalling that $\|\cdot\|_{\ell^1(N)}\leq\|\cdot\|_{\ell^2(N)}\leq\|\cdot\|_{\infty}$, we also have $|T(f)-T(\tilde{f})|\leq 7\|f-\tilde{f}\|_{\ell^2(N)}$ and $|T(f)-T(\tilde{f})|\leq 7\|f-\tilde{f}\|_\infty$.

\begin{proof}
Write $g = f - \tilde{f}$. Writing $f = \tilde{f} + g$, $T(f)$ may be expanded as a sum of $8$ terms, one of which is $T(\tilde f)$, the other 7 of which are trilinear terms of the form $\E_{n,n'} f_1(n) f_2(n') f_3(n + n')$ with at least one of the $f_i$ being equal to $g$. Using the hypothesis that all the $f_i$'s are bounded by $1$, the estimate
\[ | \E_{n, n'} f_1(n) f_2(n') f_3(n + n')| \leq \| f_i \|_{\ell^1(N)}\] is an easy consequence of the triangle inequality.

The case of the Gowers $U^2(N)$-norm is dealt with in a similar way, using instead the bound
\[ |\E_{n,n'} f_1(n) f_2(n') f_3(n + n')| \ll \| f_i \|_{U^2(N)}.\] This is an instance of a \emph{generalised von Neumann theorem}, for which there are many references including \cite[Lemma 11.4]{tv}.\end{proof}

\begin{lemma}\label{lip-conv}
Let $X$ be a compact metric abelian group endowed with a translation-invariant metric $d$ and a translation-invariant probability measure $\mu$. Suppose that $f : X \rightarrow \C$ is a function with $\| f \|_{\Lip} \leq K$, thus $|f(x) - f'(x)| \leq Kd(x,x')$. Let $g : X \rightarrow \C$ be any continuous function with $\| g \|_{\infty} \leq 1$. Then the convolution $f \ast g(x) = \int f(y) g(x - y) d\mu(y)$ also has Lipschitz constant at most $K$.
\end{lemma}
\begin{proof}
We have
\begin{align*}
f\ast g(x) - f\ast g(x') & = \int (f(x- y) - f(x' - y)) g(y) d\mu(y) \\ & \leq \int |f(x - y) - f(x' - y)| d\mu(y) \\ & \leq K \sup_y d(x - y, x' - y) = K d(x, x').\qedhere
\end{align*}
\end{proof}

The next lemma is an instance of Young's inequality, but we include the (short) proof for ease of reference.

\begin{lemma}\label{l1-conv}
Let $P, P' \subset \{1,\dots, N\}$ be arithmetic progressions with the same length. Let $f : P \rightarrow \C$ and $g : P' \rightarrow \C$ be two functions. Suppose that both are bounded pointwise by $1$ and that either $\E_{n \in P} |f(n)| \leq \eta$ or $\E_{n \in P'} |g(n)| \leq \eta$. Write $f \ast g (n) = \frac{1}{N}\sum_m f(m) g(n-m)$. Then $\| f \ast g\|_{\infty} \leq \eta|P|/N$.
\end{lemma}
\begin{proof}
Suppose that $\E_{n \in P} |f(n)| \leq \eta$. Then we have
\[ |f \ast g(n)|  = \frac{1}{N}\left|\sum_m f(m) g(n - m)\right| \leq \frac{1}{N}\sum_m |f(m)| \leq \eta |P|/N.\]
The case $\E_{n \in P} |g(n)| \leq \eta$ is similar. 
\end{proof}

\begin{lemma}\label{gowers-error}
Let $f, \tilde f,g : \{1,\dots,N\} \to [-1,1]$ be functions such that $\| \tilde f - f \|_{U^2(N)} \leq \delta$. Then $|\tilde f \ast g(d) - f \ast g(d)| \leq 4\delta^{1/2} $ for all except at most $40\delta N$ values of $d$. 
\end{lemma}
\begin{proof}
The functions $f$ and $g$ may be regarded as functions on $G = \Z/N'\Z$, where $N' = 4N$, in a natural way. Let $h = \tilde f - f$. Then
\begin{align*}
\E_{x \in G} | \E_{y \in G} h(y) g(x-y)|^2 & = \sum_r |\hat{h}(r)|^2 |\hat{g}(r)|^2 \\ & \leq \left(\sum_r |\hat{h}(r)|^4 \right)^{1/2} \left(\sum_r |\hat{g}(r)|^4 \right)^{1/2} \\ & = \| h \|_{U^2(G)}^2 \| g \|_{U^2(G)}^2 \leq \| h \|_{U^2(G)}^2 \leq 10 \delta^2.
\end{align*} Once again we have used basic facts about the $U^2(G)$-norm and the discrete Fourier transform as may be found in \cite[Section 4.2]{tv}.
Replacing the expectations over $G$ by sums we obtain
\[ \sum_x |h\ast g(x)|^2 \leq 640 \delta^2 N.\] Thus there cannot be more than $40\delta N$ values of $x$ for which $|h \ast g(x)| \geq 4 \delta^{1/2}$.
\end{proof}

\bibliography{sumfree}{}
\bibliographystyle{alpha}
\end{document}